\documentclass{article}
\usepackage{authblk}
\usepackage{graphicx}
\usepackage{tabularx}
\usepackage{amsfonts,amsmath,amssymb,amsthm}
\usepackage{bbm}
\usepackage{commath}
\usepackage[inline]{enumitem}   
\def \E {\mathbb{E}}
\def \P {\mathbb{P}}

\def \i {\mathbbm{1}}
\usepackage{mathtools}

\usepackage{setspace}
\usepackage{caption}
\usepackage{subcaption}
\usepackage{marvosym}
\usepackage{natbib}

\usepackage{graphicx}
\graphicspath{{images/}}

\usepackage{notes}

\usepackage{psfrag}
\usepackage{cases}
\usepackage{multirow}
\usepackage{booktabs}

\usepackage[all]{xy}

\input xy
\xyoption{all}

\usepackage{tikz}
\usetikzlibrary{automata,arrows,positioning,calc,decorations.pathreplacing,shapes}

\newcommand*{\Exp}{\ensuremath{\mathrm{Exp}}}

\newcommand{\mockalph}[1]{}
\usepackage{todonotes}
\newtheorem{theorem}{Theorem}

\newtheorem{proposition}{Proposition}

\usepackage{geometry}
\allowdisplaybreaks
\title{Condition based maintenance policies under imperfect maintenance at scheduled and unscheduled opportunities}
\author{C. Drent}
\author{S. Kapodistria\thanks{{Corresponding author. \emph{Email address}: s.kapodistria@tue.nl.}} }
\author{J. A. C. Resing}
\affil{{Eindhoven University of Technology, Department of Mathematics and Computer Science \\P.O. Box 513, 5600 MB Eindhoven, The Netherlands}}
\date{\today}
\begin{document}
\maketitle
\begin{abstract}
Motivated by the cost savings that can be obtained by sharing resources in a network context, we consider a stylized, yet representative, model for the coordination of maintenance and service logistics for a geographic network of assets. Capital assets, such as wind turbines in a wind park, require maintenance throughout their long lifetimes. Two types of preventive maintenance are considered: planned maintenance at periodic, scheduled opportunities, and opportunistic maintenance at unscheduled opportunities. The latter type of maintenance arises due to the network context: when an asset in the network fails, this constitutes an opportunity for preventive maintenance for the other assets in the network. 

So as to increase the realism of the model at hand and its applicability to various sectors, we consider the option of not-deferring   and of deferring  planned maintenance after the occurrence of opportunistic maintenance. We also assume that preventive maintenance may not always restore the condition of the system to `as good as new'.  By formulating this problem as a semi-Markov decision process, we characterize the optimal policy as a control limit policy (depending on the remaining time until the next planned maintenance) that indicates on the one hand when it is optimal to perform preventive maintenance and on the other hand when maintenance resources should be shared if an opportunity in the network arises. In order to facilitate managerial insights on the effect of each parameter on the cost, we provide a closed-form expression for the long-run rate of cost for any given control limit policy (depending on the remaining time until the next planned maintenance) and compare the costs (under the optimal policy) to those of sub-optimal policies that neglect the opportunity for resource sharing. We illustrate our findings using data from the wind energy industry.
\end{abstract}

\section{Introduction}
\label{sec:intro}
High valued capital assets, such as energy systems (e.g., wind turbines), medical systems (e.g., interventional X-ray machines), lithography machines in semiconductor fabrication plants, and baggage handling systems at airports require maintenance throughout their (long) lifetimes. Such capital assets are crucial to the primary processes of their users/operators and unexpected failures may have very significant negative impacts and even life threatening consequences.  In order to avoid or to minimize failures, asset owners perform preventive maintenance activities, with the objective to retain or to restore a system back to a satisfactory operating condition. The costs of both these maintenance activities, and of their respective unscheduled downtimes, represent one of the key drivers of an organization's total costs. Such maintenance costs constitute up to 70\% of the total value of the end product \citep{bevilacqua2000analytic,mobley2002introduction}, and this percentage is rapidly increasing  \citep{zio2013evaluating}. Hence, there is great incentive for asset owners to optimize the maintenance planning. 

The most common maintenance practices are the so-called {\em corrective maintenance} and the {\em planned maintenance}. The former as the name suggests proposes the repair of the asset upon failure, while the latter proposes a fixed service schedule for the field service engineers with the objective of ensuring that the asset operates correctly and of avoiding any unscheduled breakdown and downtime. The cost of planned maintenance is relatively low in comparison to that of corrective maintenance, 
due to its planned, anticipated nature. Planned maintenance is characterized by its scheduled downtimes (contrary to the unscheduled downtime experienced at a failure, which leads to a corrective maintenance) with fixed inter-scheduled instances, say at instances $\tau,2\tau,3\tau,\ldots$, (e.g., $\tau$ = 6 months). Such instances constitute the {\em scheduled opportunities} of preventive maintenance.

In the context of a network of assets, such as a wind park or a network of hospitals in close geographic proximity (from the viewpoint of the service provider), there is a second type (in addition to the above scheduled instances) of opportunities to perform preventive maintenance. In the event that a failure occurs, its corrective maintenance instance can be viewed as an unscheduled opportunity for preventive maintenance for the other assets in the network. In these instances, {\em opportunistic maintenance} can take place, with the respective instances constituting the  {\em unscheduled opportunities} of preventive maintenance. This form of network dependency can be viewed on two levels: (i) the economic dependency between the various systems of a network, and (ii) the structural degradation and failure dependencies. Similarly to planned maintenance, opportunistic maintenance has a lower cost in comparison to that of corrective maintenance. 

Incorporating opportunistic maintenance may also affect the scheduling of planned maintenance, as it might be beneficial to defer the planned maintenance opportunity to take place after a period of length $\tau$ after the occurrence of an opportunistic maintenance. This decision of {\em deferring or not the scheduling of planned maintenance} after the occurrence of opportunistic maintenance may have a positive or negative effect on the total costs.

In maintenance, it is oftentimes assumed that a maintenance activity is perfect, i.e. it restores the system to a state of `as good as new'. However, this assumption may not be true in practice. For instance, a misidentification of the root cause of the (imminent) failure can lead to an erroneous repair not resolving the actual issue, or some minor repair activity (such as exchange of parts, changes or adjustment of the settings, software update, lubrication or cleaning, etc. see \citep{spinato2009reliability}) may not restore the system to a state of `as good as new'. In the above mentioned cases, it is more reasonable to assume that the system is restored to a state between `as bad as old' and `as good as new'. This concept will be referred to as {\em imperfect maintenance}. Evidently, this assumption impacts the resulting cost. Hence, knowledge regarding the degree of how successful a maintenance activity is should not be ignored in the maintenance planning. 

In conclusion, asset owners are oftentimes faced with the following questions: 
\begin{enumerate}[label=(\roman*)]
\item What is the advantage of incorporating planned maintenance in comparison to exercising only corrective maintenance?
\item What is the benefit of sharing resources in the network (in the form of incorporating opportunistic maintenance in addition to the planned maintenance)?
\item What is the influence of deferring the planned maintenance after the occurrence of opportunistic maintenance? 
\item What is the influence of imperfect maintenance on the maintenance planning and on the costs (long-run rate of cost)? 
\item When should preventive maintenance be performed (so as to minimize the long-run rate of cost)?
\end{enumerate}

\paragraph{Main contributions}\mbox{}\\
We consider a stylized, yet representative model that incorporates the above-mentioned characteristics and we prove the existence of the optimal maintenance policy and we derive its structure. Furthermore, we compute an explicit expression for the long-run rate of cost, which can be easily used by asset owners and service providers so as to gain further insights into their practice and so as to compute the cost-benefits of changing their maintenance practice. More concretely, the main contributions of the paper are threefold: 
\begin{enumerate*}[label=(\arabic*)]
\item We consider a semi-Markov decision process that incorporates planned and opportunistic maintenance, as well as imperfect maintenance. From the analysis of the semi-Markov decision process stems the characterization of the optimal policy as a control limit policy (threshold) depending on the time until the next planned maintenance opportunity. Moreover, using this approach, we are able to derive a closed-form expression for this control limit. 
\item Considering the class of control limit  policies (depending on the remaining time until the next planned maintenance), we derive, using the theory of regenerative processes, an explicit expression for the long-run rate of cost. 
\item We consider data from the wind energy industry and provide, based on these values, concrete answers to Questions (i)--(v) mentioned above. More specifically, we analyze the benefit of using planned and opportunistic maintenance compared to only corrective maintenance. We also analyze the influence of deferring planned maintenance after the occurrence of opportunistic maintenance. Finally, we also highlight the cost savings that can be attained by reducing the probability of an imperfect maintenance. 
\end{enumerate*}

\paragraph{Outline of this paper}\mbox{}\\
The remainder of this paper is structured as follows: In Section \ref{sec:review}, we review the related literature.  In Section \ref{sec:model}, we describe in detail the model at hand, which captures the condition of the asset  and which incorporates imperfect maintenance at scheduled and unscheduled maintenance opportunities. Subsequently, in Section \ref{sec:optpolicy}, we characterize the structure of the optimal policy for condition based maintenance using the average cost criterion, see Section \ref{Sec:AverageCostCriterion}, and we compute the long-run rate of cost for any policy with the same structure as the optimal policy (i.e. the class of control limit policies depending on the remaining time until the next planned maintenance), see Section \ref{subsec:longrunaverage}. In Section \ref{sec:rescheduling}, we permit the deferral of planned maintenance after the occurrence of opportunistic maintenance, and we compute the long-run rate of cost. A numerical illustration is provided in Section \ref{sec:numerical}, where, based on data from the wind energy industry, we compare the long-run rate of cost for various policies, we show the effect of imperfect maintenance, and the effect of deferring planned maintenance. Finally, Section \ref{sec:conclusion} contains concluding remarks and highlights directions for future research. 

\section{Literature review}
\label{sec:review}
Maintenance optimization models have been extensively studied in the literature. Optimal maintenance policies aim to provide optimal system reliability/availability and safety performance at lowest possible maintenance costs \citep{pham1996imperfect}. Due to the fast development of sensing techniques in recent years, the state of a capital asset can be monitored or inspected at a much lower cost and in a continuous fashion, which facilitates condition based maintenance. Condition based maintenance  recommends maintenance actions based on information collected through online monitoring of the capital asset and it can significantly reduce maintenance costs by decreasing the number of unnecessary maintenance operations, see e.g., \cite{jardine2006,peng2010current,lam2015myopic}. The condition based maintenance model that we propose builds on the delay time model  proposed by \cite{christer1982modelling} and \cite{christer1984delay}.  We refer the reader to \cite{baker1994review, christer1999developments}, and \cite{wang2008delay}, and more recently, \cite{wang2012overview} for an overview on delay time models. Not only are delay time models well-known in literature, but they are also very frequently appearing in practice. 
%

Practice-based research with real diagnostic data, such as data related to the spectrometry of oil \cite[e.g.,][]{makis2006application,kim2011optimal} and data related to vibrations \citep[e.g.,][]{yang2010arx}, showed that it is usually sufficient, and even preferable from a modeling and decision making perspective, to consider only two operational states. The first state is the perfect state, in which the system lasts from newly installed to the point that a hidden defect has been identified. After the occurrence of a hidden defect in the system until the occurrence of a failure (which is typically referred to as the delay time), the system resides in the second state also referred to as the satisfactory state. Such a classification of the operational states has the property that maintenance actions are initiated only when the system is degraded to the state that can actually lead to a direct failure, i.e. the satisfactory state, but not when the system is functioning perfectly, i.e. the perfect state. The vast majority of the literature on delay time models is restricted to numerical methods or approximations to solve the models at hand, due to their underlying complexity. Few recent exceptions are \cite{maillart2002cost}, \cite{kim2013joint} and  \cite{van2014optimal}, who study two-state systems under periodic inspection, partial observability, and postponed replacement, respectively, and provide analytical results regarding the structure of the optimal policy. However, all of them do not consider the option of resource sharing in the network (in the form of opportunistic maintenance), nor do they incorporate the notion of imperfect repair. 

Most delay time model analyses assume that the system after a maintenance action is restored to a state of `as good as new'. Contrary to this assumption, in imperfect maintenance it is assumed that upon preventive maintenance, the system lies in a state somewhere between `as good as new' and `as bad as old'. This is first introduced by \cite{nakagawa1979optimum,nakagawa1979imperfect} and is called the $(p,q)$-rule. Under the $(p,q)$-rule,  the system is returned to an  `as good as new' state (perfect preventive maintenance) with probability $p$ and it is returned to the `as bad as old' state (minimal preventive maintenance) with probability $q = 1 - p$ after preventive maintenance. Clearly, the case $p = 0$ corresponds to having no preventive maintenance. Also, from a practical point of view, imperfect maintenance can describe a large set of realistic maintenance actions \citep{pham1996imperfect}.

When planning condition based maintenance strategies, see, e.g., \cite{jardine2006,jardine2005maintenance,prajapati2012condition}, a typical assumption in the literature is that the system at hand is monitored continuously and one can intervene and maintain the  system at any given moment. However, due to accessibility reasons (e.g., in the case of off-shore wind parks) or for cost reduction purposes, it is cost optimal and more practical to allow only for discrete time opportunities. The simplest amongst the discrete time opportunities are the periodic planned maintenance instances (also referred to as scheduled downs), with period say $\tau$, that serve as a scheduled opportunity to do maintenance for a network of systems. Furthermore, unplanned maintenance instances (due to opportunistic maintenance) can be modeled as discrete instances occurring according to a multi-dimensional counting process. 

For recent works related to opportunistic maintenance, the interested reader is referred to \cite{zhu2016age,zhu2017condition,arts2018design,Kalosi}.  In \cite{zhu2016age} and \cite{zhu2017condition}, the authors consider a single-unit system and account for both scheduled and unscheduled opportunities. In these analyses, the authors model the age and the condition, respectively, of the system and derive, based on approximations, the long-run rate of cost under a given policy. In both papers, the arrivals of unscheduled opportunities are modeled according to a homogeneous Poisson process. This approximation is justified by the Palm-Khintchine theorem  \citep{khinchin1956sequences}, which states that even if the failure times of some systems do not follow exponential distributions, the superposition of a sufficiently large number of independent renewal processes  behaves asymptotically like a Poisson process.  \cite{arts2018design} build further on \cite{zhu2016age,zhu2017condition}, but they only consider scheduled maintenance opportunities (excluding unscheduled opportunities). Furthermore, \cite{arts2018design} assume that at a scheduled opportunity, the system is restored to a perfect condition (i.e. $p=1$), while at a failure they assume that the system is restored to a state which is stochastically identical to the state just prior to the system's failure. In a recent conference paper, \cite{Kalosi} looked at a model with both planned and unplanned maintenance opportunities, at which the system is restored to a perfect condition, showing some preliminary results that a control limit  policy (depending on the remaining time until the next planned maintenance)  is optimal. 

In contrast to  \cite{arts2018design} and to \cite{zhu2016age,zhu2017condition}, in which the long-run rate of cost is computed for a given policy, we first characterize the structure of the optimal policy explicitly and thereafter, for the optimal policy class, we compute the long-run rate of cost. Furthermore, we include both scheduled and unscheduled maintenance opportunities. In contrast to \cite{Kalosi}, we extend the model by incorporating the $(p,q)$-rule, making it more generic and realistic. Moreover, we are the first to analyze the influence of deferring planned maintenance and we illustrate the financial effects of the maintenance policy in a realistic context using data stemming from the wind industry.

\section{Model description}
\label{sec:model}
We consider a single unit system (equivalently, a component or asset) that is monitored continuously and whose condition is fully observable. We assume that the condition of the system degrades over time and that it can be modeled according to a delay time model. That is, the states are classified as {\em perfect}, {\em satisfactory} and {\em failed}. We shall refer to the state of perfect condition as state $2$, the state of satisfactory condition as state $1$ and the failure state as state $0$. Furthermore, we assume that as soon as a system failure occurs, the system is instantaneously replaced by an `as good as new' system. So, in the mathematical formulation of the model, we may assume, due to the instantaneous replacement at failure, that the model evolves between only  states $1$ and $2$. The system spends an exponential amount of time with rate $\mu_i$ in state $i$, $i\in\{1,2\}$. The above model formulation implies that initially the system starts in state $2$ (perfect state), then after an exponential amount of time with rate $\mu_2$, the system deteriorates and the condition of the system goes to state $1$ (satisfactory state). The system spends an exponential amount of time with rate $\mu_1$ in state $1$, after which a failure occurs. At a failure the system is instantaneously replaced by an `as good as new' system and the condition is restored to $2$ (perfect state).  A schematic evolution of the condition of the component and the corresponding times of transitions are depicted in Figure \ref{Fig:1}.

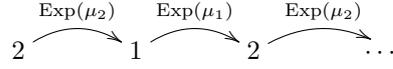
\begin{figure}[h]
\[\def\labelstyle{\scriptstyle}
 \xymatrix  @C+9pt@R+1pt{
&
2  \ar@/^.75pc/[r]^{\mathrm{Exp}(\mu_2)}&1\ar@/^.75pc/[r]^{\mathrm{Exp}(\mu_1)}&2  \ar@/^.75pc/[r]^{\mathrm{Exp}(\mu_2)}&\cdots\ \ \ \\}
\]
\caption{Schematic evolution of the condition of the component and the corresponding times of transitions.\label{Fig:1}}
\end{figure}

We assume that we have two types of opportunities in which we can perform preventive maintenance (PM) before failure: the  scheduled and the unscheduled opportunities. The scheduled opportunities correspond to pre-arranged opportunities occurring according to a fixed schedule. These opportunities can be attributed to either  service/maintenance agreements or to regulation imposition checks. We assume that the scheduled opportunities occur at epochs $\tau,2\tau,3\tau,\ldots$, with $\tau>0$. This is also in accordance with what happens in practice as maintenance actions once planned are typically not rescheduled. The unscheduled opportunities correspond to random opportunities triggered by failures of other systems in close proximity. We assume that these unscheduled opportunities occur according to a Poisson process at rate $\lambda$.

The unscheduled and scheduled opportunities, abbreviated by USO and SO, respectively, serve as opportunities to perform preventive maintenance. Such a preventive maintenance is assumed to cost less than a corrective maintenance (CM) upon failure, which costs $c_{\text{cm}}$. Moreover, incorporating a planning perspective, we may assume that the preventive maintenance cost at an SO, $c_{\text{pm}}^{\text{so}}$, is less than or equal to the corresponding cost at a USO, say $c_{\text{pm}}^{\text{uso}}$, that is $0<c_{\text{pm}}^{\text{so}}\le c_{\text{pm}}^{\text{uso}}<c_{\text{cm}}$ (however, we also extend our analysis to the case $c_{\text{pm}}^{\text{so}}> c_{\text{pm}}^{\text{uso}}$). Following the $(p,q)$-rule of \cite{nakagawa1979imperfect,nakagawa1979optimum}, we assume that  after preventive maintenance a system is returned to the `as good as new' state with probability $p\in (0,1]$ and returned to the `as bad as old' state (i.e. the amount of time left until the failure has not altered) with probability $q=1-p$.

Our aim is to determine a policy when to perform preventive maintenance on the system based on its condition and the opportunity type, i.e.  scheduled or unscheduled. More explicitly, we will need to formally define the state space, which refers to the condition of the system, the action space and the decision epochs. The state space is governed by the process depicting the condition of the system, i.e.  the Markov chain evolving between the states $\{1,2\}$. The action space consists of only two actions: perform preventive maintenance  or do nothing. Lastly, the decision epochs are the SO and USO epochs. In Figure \ref{Fig:2}, we depict the SO epochs by ($*$) and the USO epochs by (o).

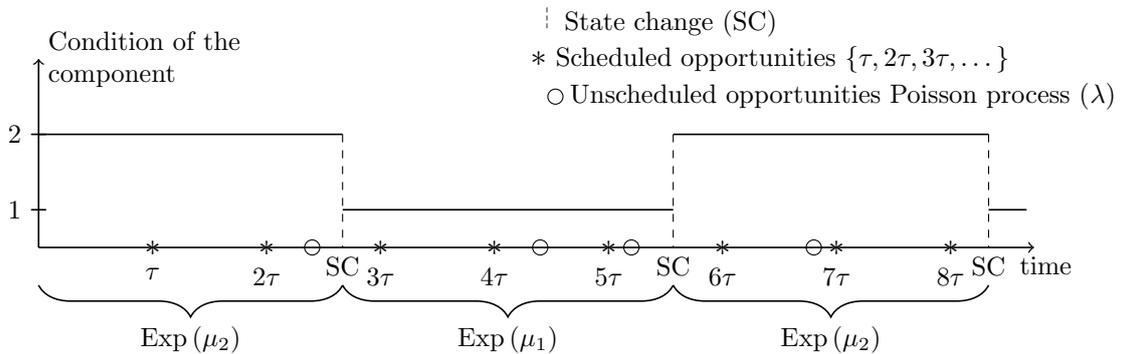
\begin{figure}[ht!]
\begin{center}
		\begin{tikzpicture}
		\draw [semithick,->] (-6,0) -- (-6,2.5) node[right,text width=2.7cm]{Condition of the component};
		\draw [semithick] (-6.1 ,1.5) node[left]{2} -- (-5.9,1.5);
		\draw [semithick] (-6.1 ,.5) node[left]{1} -- (-5.9,.5);
		\draw [semithick,->] (-6,0) -- (7.1,0) node[below,xshift=1ex]{time};
		\draw [semithick] (-6,1.5) -- (-2,1.5);
		\draw  (-4.5,0) node[yshift=-.2pt]{\large $\ast$} node[below,yshift=-1ex]{ $\tau$};
		\draw  (-3,0) node[yshift=-.2pt]{\large $\ast$} node[below,yshift=-1ex]{ $2\tau$};
		\draw  (-2.4,0) circle (0.1);				
		\draw [semithick] (-2,.5) -- (2.35,.5);				
		\draw [dashed] (-2,1.5) -- (-2,0) node[below]{SC};								
		\draw  (-1.5,0) node[yshift=-.2pt]{\large $\ast$} node[below,yshift=-1ex]{ $3\tau$};
		\draw  (0,0) node[yshift=-.2pt]{\large $\ast$} node[below,yshift=-1ex]{ $4\tau$};
		\draw  (.6,0) circle (0.1);					
		\draw  (1.5,0) node[yshift=-.2pt]{\large $\ast$} node[below,yshift=-1ex]{ $5\tau$};				
		\draw  (1.8,0) circle (0.1);					
		\draw [dashed] (2.35,1.5) -- (2.35,0) node[below]{SC};				
		\draw [semithick] (2.35,1.5) -- (6.5,1.5);						
		\draw  (3,0) node[yshift=-.2pt]{\large $\ast$} node[below,yshift=-1ex]{ $6\tau$};		
		\draw  (4.2,0) circle (0.1);							
		\draw  (4.5,0) node[yshift=-.2pt]{\large $\ast$} node[below,yshift=-1ex]{ $7\tau$};				
		\draw  (6,0) node[yshift=-.2pt]{\large $\ast$} node[below,yshift=-1ex]{ $8\tau$};				
		\draw [dashed] (6.5,1.5) -- (6.5,0) node[below]{SC};	
		\draw [semithick] (6.5,.5) -- (7,.5);
		
		\draw[semithick,decorate,decoration={brace,amplitude=12pt,mirror}] 
		(-6,-.5) -- (-2,-.5)  node [midway,below,yshift=-12pt]  {$\Exp\del{\mu_2}$};
		\draw[semithick,decorate,decoration={brace,amplitude=12pt,mirror}] 
		(-2,-.5) -- (2.35,-.5)  node [midway,below,yshift=-12pt]  {$\Exp\del{\mu_1}$};				
		\draw[semithick,decorate,decoration={brace,amplitude=12pt,mirror}] 
		(2.35,-.5) -- (6.5,-.5)  node [midway,below,yshift=-12pt]  {$\Exp\del{\mu_2}$};
		\node at (3.625,2.5) {{\large{$\ast$}} Scheduled opportunities $\cbr{\tau,2\tau,3\tau,\dots}$};
		\draw  (.8,2) circle (0.1) node[right,xshift=.45ex]{Unscheduled opportunities Poisson process $(\lambda)$};					
		\node at (2.2,3) {{\large{\rotatebox{90}{\Kutline}}} State change (SC)};
		\end{tikzpicture}
\end{center}
\caption{A sample path of the model.\label{Fig:2}}
\end{figure}

Table \ref{tab:abbrev} summarizes the abbreviations that we will use throughout the remainder of this paper. 
\begin{table}[!htb]
\centering
\begin{tabular}{c l} 
\toprule
PM & Preventive maintenance\\
CM & Corrective maintenance\\
USO & Unscheduled opportunity\\
SO & Scheduled opportunity\\
SC & State change \\
\bottomrule
\end{tabular}
\caption{Overview of abbreviations.}
\label{tab:abbrev}
\end{table}


\section{Optimal policy}
\label{sec:optpolicy}
The goal of this section is twofold: We first characterize the structure of the optimal average cost condition based maintenance policy. We then derive an explicit form for the long-run rate of cost per time unit for any given policy that has the same structure as the optimal policy.

\subsection{Average cost criterion}
\label{Sec:AverageCostCriterion}
This section is devoted to the derivation of the optimal policy on when to perform preventive maintenance for the system at hand using the average cost criterion. To this purpose, we set up our problem as a (controlled) semi-Markov decision process. Due to the stochastic nature of the problem, it does not suffice to know the type of the decision epoch (SO or USO), but it is also required to keep track of the remaining time till the next SO. That time may impact our decision, i.e.  the optimal policy may depend on the residual time till the next SO. Thus, for the full description of the condition (state) of the system, we use a triplet descriptor 
$$\mathcal{S}=\left\{(i,j,t):\ i\in \{1,2\}, \ j\in\{\text{SC},\text{USO}\},\ t\in(0,\tau)\right\}\cup \left\{(i,\text{SO},0):\ i\in \{1,2\}\right\},$$
where $i$ indicates the condition of the system. If $j=\text{SC}$, then this means that the condition of the system is about to change and there is no decision associated with this epoch, while if $j=\text{SO}$ or $j=\text{USO}$, this means that this is a decision moment at either a scheduled (SO) or unscheduled opportunity (USO), respectively. Finally, the third element indicates the remaining time until the SO. Note that if  $j=\text{SO}$ then $t=0$. The introduction of the remaining time until the upcoming SO in the full description of the condition of the system renders the model inhomogeneous, and for this reason we  use techniques that stem from semi-Markov decision processes. Note here that the inclusion of the remaining time until the upcoming SO in the state, although it complicates the analysis, permits us to prove that there is an optimal policy in the class of deterministic stationary policies, cf. Propositions \ref{prop:ACOI} and \ref{prop:ACOI-full}. At each decision epoch (depending on the values of $(i,j,t)\in\mathcal{S}$), we can choose to perform preventive maintenance or do nothing or in case of a failure to do corrective maintenance (CM), that is $\mathcal{A} =\{\text{perform PM, do nothing, perform CM}\}$, where $\mathcal{A}$ represents the overall action space. 

\begin{proposition}\label{prop:ACOI}
For the model at hand, the deterministic stationary policy is optimal for the average cost criterion.
\end{proposition}

A formal version of the above proposition, cf. Proposition \ref{prop:ACOI-full}, and its proof can be found in Appendix \ref{sec:appendixmdp}, together with a full formal definition of the model in the context of semi-Markov decision processes. In addition to the theoretical validation that the above proposition offers on the existence and nature of the optimal maintenance policy, in the following theorem we compute the optimal policy.

\begin{theorem}\label{thm:opt1}
Under the assumption that $c_{\text{pm}}^{\text{so}}< c_{\text{pm}}^{\text{uso}}$ and given the imperfect preventive maintenance probability $1-p\in(0,1]$,   the optimal policy under the average cost criterion is: For state $2$ to do nothing. For state $1$ to perform preventive maintenance at scheduled opportunities, if $ \mu_1 c_{\text{cm}} > (\mu_1+\mu_2)\frac{c_{\text{pm}}^{\text{so}}}{p}$, and to do nothing otherwise, and to perform preventive maintenance at unscheduled opportunities for which the residual time until the next scheduled opportunity is in $[\hat{t},\tau)$, if $\mu_1 c_{\text{cm}} > \left( \frac{c_{\text{pm}}^{\text{uso}}}{p} - \frac{c_{\text{pm}}^{\text{uso}} - c_{\text{pm}}^{\text{so}}}{e^{(\mu_1 +\mu_2)\tau}-1} \right) (\mu_1+\mu_2)$, and to do nothing otherwise. Where, $\hat{t} = \min\{\tau,\max\{0,t^*\}\}$, with $t^*$ satisfying
\begin{align}
\frac{c_{\text{pm}}^{\text{uso}}}{p} &= \frac{\mu_1c_{\text{cm}} + \lambda c_{\text{pm}}^{\text{uso} }}{\mu_1+\mu_2 + \lambda p } + \left(
\frac{-c_{\text{pm}}^{\text{so}}+\frac{\mu_1c_{\text{cm}}}{\mu_1+\mu_2} + \left( \frac{c_{\text{pm}}^{\text{uso}}}{p} - \frac{\mu_1 c_{\text{cm}}}{\mu_1 + \mu_2} \right) e^{(\mu_1+\mu_2)t^*}}{1-p}
 -\frac{\mu_1c_{\text{cm}} + \lambda c_{\text{pm}}^{\text{uso} }}{\mu_1+\mu_2 + \lambda p } \right) e^{(\mu_1+\mu_2 + \lambda p )(\tau-t^*)}. \label{thm1eq1}
\end{align}
\end{theorem}
\noindent \begin{proof}
See Appendices \ref{sec:BellmanEq} and \ref{proofTheorem1}. \end{proof}

For USOs, Theorem \ref{thm:opt1} establishes a control limit policy depending on the remaining time until the next SO: if the residual time until the next SO is smaller than $\hat{t}$, then it is optimal to not take the opportunity to perform preventive maintenance in state 1. This is intuitive in the sense that the urgency for preventive maintenance in state 1 at a USO should decrease as the cheaper opportunity at an SO is approaching. 

Note that in the special case when preventive maintenance costs at SOs and USOs are equal, the optimal policy reduces to a stationary control limit policy, which is shown in Proposition \ref{prop:equal}.
\begin{proposition}\label{prop:equal}
Under the assumption that $c_{\text{pm}}^{\text{so}}=c_{\text{pm}}^{\text{uso}}=c_{\text{pm}}>0$ and given the imperfect preventive maintenance probability $1-p\in(0,1]$,  the optimal policy under the average cost criterion is: For state $2$ to do nothing. For state $1$ to perform preventive maintenance at both SOs and USOs, if $ \mu_1  c_{\text{cm}}> (\mu_1+\mu_2) \frac{c_{\text{pm}}}{p}$, and to do nothing otherwise.
\end{proposition}
\noindent \begin{proof}
The proof of this proposition is identical in structure to the proof of Case (i) in Theorem \ref{thm:opt1} and for this reason it is omitted.  
\end{proof}

One could also argue that the cost for preventive maintenance at a USO is actually less than the cost at an SO since there is already a cost attached to the opportunity at hand (e.g., service engineers are already at a wind park and they can at a small extra cost repair other systems in close proximity as well). In this case, the optimal control policy also reduces to a stationary control limit policy, which is described in Theorem \ref{prop:unequal}. 
\begin{theorem}\label{prop:unequal}
Under the assumption that $c_{\text{pm}}^{\text{so}}>c_{\text{pm}}^{\text{uso}}$ and given the imperfect preventive maintenance probability $1-p\in(0,1]$,  the optimal policy under the average cost criterion is: For state 2 to do nothing. For state 1 to perform preventive maintenance at an unscheduled opportunity if $ \mu_1  c_{\text{cm}} > (\mu_1+\mu_2)\frac{ c_{\text{pm}}^{\text{uso}}}{p} $, and to do nothing otherwise, and to perform preventive maintenance at an SO if  $\mu_1  c_{\text{cm}} > (\mu_1+\mu_2 )\frac{c_{\text{pm}}^{\text{so}}}{p}+ \lambda ({c_{\text{pm}}^{\text{so}}}-c_{\text{pm}}^{\text{uso}})$, and to do nothing otherwise. 
\end{theorem}
\noindent \begin{proof}
See Appendix \ref{proofTheorem2}. \end{proof}

\subsection{Long-run rate of cost per time unit}
\label{subsec:longrunaverage}
In the previous section, we characterized the structure of the optimal policy using the average cost criterion. This policy can be viewed as a control limit policy, with the control limit depending on the time until the next SO. In this section, we consider such a policy and we compute the long-run rate of cost per time unit. More concretely, we consider a policy under which in state 2 we do not perform preventive maintenance (i.e. we do nothing), and in state 1 we  always perform preventive maintenance at SOs and we perform preventive maintenance at USOs if the remaining time till the next SO is greater than $\tilde{t}$, for some given value $\tilde{t}\in(0,\tau)$. The results obtained in this section are directly applicable to the results of Section \ref{Sec:AverageCostCriterion}, by  setting $\tilde{t}=t^*$, cf. Theorem \ref{thm:opt1}.

For the computation of the long-run rate of cost per time unit, we employ  the theory of regenerative-like processes, also called stationary-cycle processes, described in Section 2.19 of \cite{serfozo2009basics}. To this purpose, we consider the inter-regeneration times created by the SOs $\{\tau, 2\tau, 3\tau, \ldots\}$. For the cost computation, we assume that, at the SOs, the system is in state 1 or 2 according to a stationary probability $p_1(0)$ and $p_2(0)$, respectively.  The long-run rate of cost per time unit is calculated as the expected total cost incurred between consecutive SOs divided by $\tau$.

Let $p_i(t)$ be the probability that the system is in state $i \in \{1,2\}$ given that the time until the next SO is $t\in[0,\tau)$, then the long-run rate of cost per time unit for this control limit policy (depending on the remaining time until the next planned maintenance) for any given time threshold is given in the next theorem. 

\begin{theorem}\label{them:costs}
Consider a given policy under which in state $2$ we opt for the action do nothing, and in state $1$ we repair at scheduled opportunities and at unscheduled opportunities for which the remaining time until the next scheduled opportunity is greater than $\tilde{t}\in(0,\tau)$, and we do nothing otherwise. Under this policy, the long-run rate of cost per time unit equals 
\begin{equation}
\frac{c_{\text{pm}}^{\text{so}}p_1(0)+c_{\text{pm}}^{\text{uso}}\lambda \int_{\tilde{t}}^{\tau}p_1(t)\dif t + c_{\text{cm}} \mu_1 \int_{0}^{\tau} p_1(t) \dif t}{\tau},\label{Theoretical_cost}
\end{equation}
with 
\begin{numcases}{ p_1(t) = }
\frac{\mu_2}{\mu_1+\mu_2}   + C_1 \, e^{(\mu_1+\mu_2)t},\ t\in [0,\tilde{t}), \label{diff1} \\
\frac{\mu_2}{\mu_1+\mu_2 + \lambda p } + C_2  \, e^{(\mu_1+\mu_2 +\lambda p)t}, \  t\in [\tilde{t},\tau), \label{diff2}
\end{numcases}
where the constants $C_1$ and $C_2$ are obtained as follows
\begin{align*}
C_1&=C_2\,  e^{\lambda p \tilde{t}}-\frac{\mu_2 }{\mu_1+\mu_2}\frac{\lambda p}{\mu_1+\mu_2 + \lambda p }e^{-(\mu_1+\mu_2)\tilde{t}},\\
C_2&=\frac{\frac{\mu_2}{\mu_1+\mu_2 } \left(1-e^{-(\mu_1+\mu_2) \tilde{t}}\right)
+ \frac{\mu_2}{\mu_1+\mu_2 + \lambda p }\left(\frac{1}{1-p}-e^{-(\mu_1+\mu_2) \tilde{t}}\right)}
{\frac{1}{1-p}e^{(\mu_1+\mu_2 + \lambda p )\tau}-e^{\lambda p \tilde{t}}}.
\end{align*}
\end{theorem}
\noindent \begin{proof} The expected total cost incurred in one cycle consists of three parts (cf. Equation \eqref{Theoretical_cost}), which are related to the expected cost associated with preventive maintenance at SOs, with preventive maintenance at  USOs and with corrective maintenance, respectively. It is now sufficient to derive $p_i(t)$ for $t\in[0,\tau)$, $i \in \{1,2\}$. 

For $t\in [\tilde{t},\tau)$, the time-dependent behavior of $p_1(t)$ is governed by  
\begin{align}
p_1(t) &= p_1(t+ \dif t)(1-(\mu_1+\lambda p)\dif t) + p_2(t+\dif t)\mu_2 \dif t. \label{timedepeq}
\end{align}
Equation \eqref{timedepeq} is easily obtained by considering a small time interval of length $\dif t$, and noticing that at time $t$ we are in state 1 either due to a transition from state 2 with infinitesimal probability $\mu_2 \dif t$ or we have remained in state 1 with infinitesimal probability $1-(\mu_1+\lambda p)\dif t$. 
Subtracting $p_1(t+\dif t)$ from both sides of Equation \eqref{timedepeq}, after some straightforward computations, yields $$p_1(t+\dif t)-p_1(t) = p_1(t+\dif t)(\mu_1+\lambda p)\dif t - p_2(t+\dif t)\mu_2 \dif t.$$
Dividing this expression by $\dif t$ and letting $\dif t\rightarrow 0$ results in $$
p_1^{'}(t) = p_1(t)(\mu_1+\lambda p) - p_2(t)\mu_2.$$
Following a similar analysis for $p_2(t)$, yields the following system of differential equations, for $t\in [\tilde{t},\tau)$,
\begin{align}
& \left[ \begin{array}{c} p_1'(t) \\ p_2'(t) \end{array} \right] = \begin{bmatrix} \mu_1 + \lambda p & -\mu_2 \\ -(\mu_1 + \lambda p) & \mu_2 \end{bmatrix} \times  \left[ \begin{array}{c} p_1(t) \\ p_2(t) \end{array} \right], & t\in [\tilde{t},\tau). \label{diff2sol} 
\end{align} 
Similarly, for $t\in [0,\tilde{t})$, we have 
\begin{align}
& \left[ \begin{array}{c} p_1'(t) \\ p_2'(t) \end{array} \right] = \begin{bmatrix} \mu_1  & -\mu_2 \\ -\mu_1 & \mu_2 \end{bmatrix} \times  \left[ \begin{array}{c} p_1(t) \\ p_2(t) \end{array} \right], & t\in [0,\tilde{t}). \label{diff1sol} 
\end{align} 
Solving the system of differential Equations \eqref{diff2sol} and \eqref{diff1sol} leads to the desired solutions \eqref{diff1} and \eqref{diff2}, respectively. In this process, we would need to compute four unknown constants. This is achieved by using: (i) the normalizing condition, i.e. $p_1(t)+p_2(t)=1$ for all $t\in[0,\tau)$, (ii) the continuity condition at $\tilde{t}$, i.e.  $\lim \limits_{t\to\tilde{t}^-} p_i(t)=p_i(\tilde{t})$ for $i\in \{1,2\}$, and (iii) the boundary condition at the SOs imposed by the policy and the imperfect maintenance probability, i.e. $(1-p)p_1(0) = \lim \limits_{t\to\tau^-}p_1(t)$.
\end{proof}

\subsubsection{Special cases} 
\label{subsec:specialcases}
In case of only \emph{scheduled opportunities}, which corresponds to the case $\tilde{t}\rightarrow \tau$ or, equivalently, to the case $\lambda \rightarrow 0$, the probabilities $p_i(t)$ for $i \in \{1,2\}$ are derived from the system of linear equations in \eqref{diff1sol} plus the normalizing condition, i.e. $p_1(t)+p_2(t)=1$ for all $t\in[0,\tau)$. This yields 
$$ p_1(t) = \frac{\mu_2}{\mu_1+\mu_2} \left(1 - \frac{p e^{(\mu_1+\mu_2)t}}{e^{(\mu_1+\mu_2)\tau}-1+p}\right),\ t\in [0,\tau). $$
Plugging the above result into Equation \eqref{Theoretical_cost}, after appropriately considering in Equation \eqref{Theoretical_cost} only the costs related to preventive maintenance at SOs and corrective maintenance
\begin{equation*}
\frac{c_{\text{pm}}^{\text{so}}p_1(0)+ c_{\text{cm}} \mu_1 \int_{0}^{\tau} p_1(t) \dif t}{\tau},
\end{equation*}
leads to the long-run rate of cost per time unit in the case of only SOs.  

In case of \emph{perfect maintenance}, i.e.  in case $p=1$,  
the boundary condition at the SOs imposed by the policy and the imperfect maintenance in the proof of Theorem \ref{them:costs} reduces to $\lim \limits_{t\to\tau^-}p_1(t)=0$, as immediately after an SO, the system is restored to state 2 with probability 1. This enables us to explicitly solve the system of linear Equations \eqref{diff2sol} and \eqref{diff1sol}, yielding
\begin{align*}
p_1(t)=&\frac{\mu_2}{\mu_1+\mu_2} + \del{\frac{\mu_2}{\lambda+\mu_1+\mu_2}-\frac{\mu_2}{\mu_1+\mu_2}-\frac{\mu_2}{\lambda+\mu_1+\mu_2}e^{\del{\lambda+\mu_1+\mu_2}(\tilde{t}-\tau+\frac{\Lambda(t)}{\lambda}(t-\tilde{t}))}}e^{\frac{\lambda-\Lambda(t)}{\lambda}(\mu_1+\mu_2)(t-\tilde{t})},
\end{align*}
where
\begin{align*}
\Lambda(t)=\begin{cases}
0,&\text{ if } \ 0\leq t<\tilde{t},\\
\lambda,&\text{ if }\ \tilde{t}\leq t<\tau.
\end{cases}
\end{align*}
Combining this expression with Equation \eqref{Theoretical_cost}, results in the long-run rate of cost per time unit in the case of perfect maintenance.

In case of only {\em unscheduled opportunities}, which is equivalent to considering  $\tau\to\infty$, the condition of the system can be fully described using a double descriptor
$\mathcal{S}=\left\{(i,j):\ i\in \{1,2\}, \ j\in\{\text{SC},\text{USO}\}\right\}$ which is independent of time, and thus the new model formulation falls into the framework of regular Markov decision processes. It can be easily shown that:
For state $2$, the optimal policy is to do nothing, and, for state 1, the optimal policy is  to repair if
$\frac{(\mu_1+\mu_2)c_{\text{pm}}^{\text{uso}}}{p} < \mu_1 c_{\text{cm}}$
and to do nothing otherwise. Furthermore, under the optimal policy the  average long-run rate of cost is equal to
$$ \frac{c_{\text{pm}}^{\text{uso}}\lambda\mu_2+c_{\text{cm}}\mu_1\mu_2}{\lambda p +\mu_1+\mu_2}.$$

In case of only {\em corrective replacements}, the long-run rate of cost is equal to
$$c_{\text{cm}}\frac{\mu_1\mu_2}{\mu_2+\mu_1}.$$

\section{Deferring planned maintenance}
\label{sec:rescheduling}
In this section, we consider that upon a successful maintenance activity (preventively, at an SO or at a USO, or correctively), the upcoming planned maintenance is deferred for a period of length $\tau$, i.e. at the instances of successful maintenance the remaining time till the next SO is set equal to $\tau$. We are interested in computing the long-run rate of cost under deferred maintenance and, in Section \ref{sec:Num_Def_M}, using the results of this section and of the previous sections, in investigating the economical benefits of deferring planned maintenance.

Analogously to the analysis of Section \ref{subsec:longrunaverage}, we derive the long-run rate of cost using renewal theory, see, e.g., \cite[Proposition 7.3, page 433]{ross2014introduction}. In this case, we consider the renewal points to be the instances at which there was a successful maintenance activity, i.e. the SOs or USOs at which the preventive maintenance was perfect, or the epochs at which corrective maintenance is performed. Note that the underlying stochastic process that governs the condition of the system, regenerates after each successful maintenance activity.  That is, after each successful maintenance activity the underlying stochastic process is in state 2 with probability 1. The long-run rate of cost per time unit for a policy in the class of optimal policies is given in the next theorem. As the expressions appearing in the theorem do not simplify upon further computations, we choose to present them in the form of probabilities and expectations associated with the exponential distribution, as these expressions are straightforward (though cumbersome to compute) and shed insight on each of the individual events participating in the final expression, cf. Equation \eqref{Eq:longrunrateofcostdef}. 
\begin{theorem}\label{them:costsDeferring}
Consider a given policy under which in state $2$ we do nothing, and in state $1$ we repair at scheduled opportunities and at unscheduled opportunities for which the remaining time until the next scheduled opportunity is greater than $\tilde{t}\in(0,\tau)$, and we do nothing otherwise. Furthermore, consider that planned maintenance is deferred after a successful maintenance. Under this setting, the long-run rate of cost per time unit equals 
\begin{align}\label{Eq:longrunrateofcostdef}
\frac{\E\left[\text{Total cycle cost}\right]}{\E\left[\text{Total cycle length}\right]} = \frac{\E\left[C\!C\right]}{ \frac{1}{\mu_2} + \E\left[C\!L\right]}= \frac{\E\left[C\!C\, \i_{\{C\!L\,  \leq Y\}}\right]+ \E\left[C\!C\,\i_{\{C\!L\,  > Y\}} \right]}{ \frac{1}{\mu_2} + \E\left[C\!L\,  \i_{\{C\!L\, \leq Y\}}\right] + \E\left[C\!L\,  \i_{\{C\!L\, >Y\}}\right]},
\end{align}
with
\begin{align}
\E\left[C\!L\,  \i_{\{C\!L\,  \leq Y\}} \right]  &=
\E\left[C\!L\,  \i_{\{\text{USO}{[\tau- Y,\tau-\tilde{t}]}\}}\right] + \E\left[C\!L\,  \i_{\{\text{SO}{[\tau-Y,\tau]}\}}\right] + \E\left[C\!L\,  \i_{\{\text{CM}{[\tau- Y,\tau]}\}}\right], \label{eclFirstPart_thm}\\
\E\left[C\!L\,  \i_{\{C\!L\, >Y\}}\right] &= (1-p)\P\left[ \text{SO}{[\tau-Y,\tau]} \right] \Bigg( \E[Y] +\frac{\tau (1-p)  \P\left[\text{SO}{[0,\tau]} \,\right] }{1-(1-p)  \P\left[\text{SO}{[0,\tau]} \,\right] } \nonumber \\
&\quad + \E\left[C\!L'\,  \i_{\{C\!L'\leq Y\}} \,|\, Y=\tau \right] \Bigg),\label{EXP2_thm}\\
 \E\left[C\!C\, \i_{\{C\!L\,  \leq Y\}}\right] &= \E\left[C\!C\,\i_{\{\text{USO}{[\tau-Y,\tau-\tilde{t}]}\}}\right] + \E\left[ C\!C\, \i_{\{ \text{SO}{[\tau-Y,\tau]}\}}\right] + \E\left[C\!C\,\i_{\{\text{CM}{[\tau-Y,\tau]}\}}\right] , \label{ECCfirstPart_thm}\\
\E\left[C\!C\,\i_{\{C\!L\,  > Y\}}\right]&= (1-p) \P \left[ \text{SO}{[\tau-Y,\tau]} \right] \Bigg(\E\left[C\!C\,\i_{\{\text{SO}{[\tau-Y,\tau]}\}}\right] \nonumber \\ 
&\quad +\frac{(\lambda (1-p)(\tau -\tilde{t})c_{\text{pm}}^{\text{uso}} +c_{\text{pm}}^{\text{so}} ) (1-p)  \P\left[\text{SO}{[0,\tau]} \,\right] }{1-(1-p)  \P\left[\text{SO}{[0,\tau]} \,\right] } + \E\left[C\!C\,\i_{\{C\!L'\,  \leq Y \}}\,|\, Y=\tau\right] \Bigg),\label{costECC}
\end{align}
where the density of the truncated exponential random variable  $Y$ is given by
\begin{align}
f_Y(y)&=\mu_2 \frac{ e^{-\mu_2(\tau-y)}}{1-e^{-\mu_2 \tau}}, \ y\in[0,\tau), \label{residualTime_thm}
\end{align}
and with, for $0\leq y\leq \tau$, 
\begin{align}
\i_{\{ \text{SO}{[\tau-y,\tau]}\}} &\stackrel{d}{=}\i_{\{y<\min\{T_{\lambda p},T_{\mu_1}\}\}}+\i_{\{T_{\lambda p}< y<\min\{T_{\mu_1},\tilde{t}\}\}}\i_{\{y< \tilde{t}\}}\nonumber\\
&\quad+\i_{\{y-\tilde{t}\leq T_{\lambda p}<y,  y\leq T_{\mu_1}\}} \i_{\{y\geq \tilde{t}\}}, \label{ISO_thm}\\
\i_{\{\text{USO}{[\tau-y,\tau-\tilde{t}]}\}} &\stackrel{d}{=}\i_{\{T_{\lambda p}<\min\{T_{\mu_1},y-\tilde{t}\}\}}\i_{\{y\geq \tilde{t}\}},  \label{IUSO_thm}\\\
\i_{\{\text{CM}{[\tau-y,\tau]}\}} &\stackrel{d}{=} 
\i_{\{T_{\mu_1}<\min\{y,T_{\lambda p}\}\}}+\i_{\{T_{\lambda p}<T_{\mu_1}<y\}}\i_{\{y<\tilde{t}\}} \nonumber\\
&\quad+\i_{\{T_{\lambda p}<T_{\mu_1}<y,\, T_{\lambda p}\geq y-\tilde{t}\}}\i_{\{y\geq\tilde{t}\}},\label{ICM_thm}\\
%
%
\E\left[C\!L\,  \i_{\{\text{USO}{[\tau- y,\tau-\tilde{t}]}\}}\right]  &= \E\left[ T_{\lambda p} \i_{\{\text{USO}{[\tau- y,\tau-\tilde{t}]}\}}   \right],  \label{USOCE_thm}\\
\E\left[C\!L\,  \i_{\{\text{SO}{[\tau- y,\tau]}\}}\right]  &= y p \P\left[\text{SO}[\tau-y,\tau\,]\right],\label{SOCE_thm}\\
\E\left[C\!L\,  \i_{\{\text{CM}{[\tau- y,\tau]}\}}\right] &= \E[T_{\mu_1}\i_{\{\text{CM}[\tau-y,\tau]\}}] ,
\label{CMCE_thm}\\
\E\left[C\!C\,\i_{\{\text{USO}{[\tau-y,\tau-\tilde{t}]}\}}\right] &= c_{\text{pm}}^{\text{uso}} \,\P\left[\text{USO}{[\tau-y,\tau-\tilde{t}]}\,\right]+ \lambda(1-p) c_{\text{pm}}^{\text{uso}}  \E\left[T_{\lambda p} \i_{\{\text{USO}{[\tau-y,\tau-\tilde{t}]}\}} \right], \label{costUSO} \\
\E\left[C\!C\,\i_{\{\text{SO}{[\tau-y,\tau]}\}}\right] &= \Big(c_{\text{pm}}^{\text{so}} + \lambda(1-p) c_{\text{pm}}^{\text{uso}}  \max\left\{y-\tilde{t},0\right\} \Big)\P\left[\text{SO}{[\tau-y,\tau]}\,\right], \label{costSO} \\ 
\E\left[C\!C\,\i_{\{\text{CM}{[\tau-y,\tau]}\}}\right] &= c_{\text{cm}} \P\left[\text{CM}{[\tau-y,\tau]}\right] \nonumber \\ &\quad + \lambda(1-p) c_{\text{pm}}^{\text{uso}} \E\left[\min\left\{ T_{\mu_1},\max\left\{y-\tilde{t},0\right\} \right\} \,\i_{\{\text{CM}{[\tau-y,\tau]}\}} \right], \label{costCM}
\end{align}
where  $\i_{\{x\}}$ is an indicator function taking value 1 if event $x$ occurs, and it is zero otherwise, 
$T_{\mu_1}\sim\text{Exp}(\mu_1)$, $T_{\lambda p}\sim\text{Exp}(\lambda p)$,  $\P \left[\,\cdot\,\right] =\E[\i_{\{\cdot\}}]$ for all events in Equations \eqref{ISO_thm}--\eqref{ICM_thm}, and $C\!L\stackrel{d}{=}C\!L'$.
\end{theorem}
\noindent \begin{proof}
See Appendix \ref{proofTheorem4}. \end{proof}

\section{Numerical results} 
\label{sec:numerical}
Using the results and the analyses of the previous sections, in this section, we illustrate through a few well chosen examples the effect of the various parameters in the long-run rate of cost. In these examples, we investigate the financial advantage of the optimal policy, when compared to other (suboptimal) policies. Furthermore, we highlight the financial benefit of perfect maintenance by comparing the long-run rate of cost for the perfect maintenance model ($p=1$) to that of the imperfect maintenance model ($p\in(0,1)$). Here, we also show the influence of imperfect maintenance on the maintenance planning. In addition, we illustrate the change introduced by the action of deferring planned maintenance after the occurrence of a successful maintenance. To illustrate the financial effects in a realistic context and to connect our analysis with the practice, we use values and data stemming from the wind industry.

\subsection{Comparison of the optimal policy to suboptimal policies}
In this section, we compute, in the context of the wind industry example, the long-run rate of cost under the optimal policy and we examine how it is affected by varying one by one the parameters $\tau$, $\lambda$ and $c_{\text{pm}}^{\text{uso}}$, while keeping all other parameters fixed. For the determination of the values used in the numerical computations of this section, we consider the gearbox of a wind turbine. Statistics from a recent field study by \cite{ribrant2007survey} on Swedish wind parks in the period 1997-2005 showed that the gearbox is the most critical unit of a wind turbine. The notion of criticality is determined by the fact that a failure of the gearbox leads to the highest downtime when compared to all other wind turbine components, but also by the fact that this component has the highest failure rate among all wind turbine components \citep{ribrant2007survey,tavner2007reliability,spinato2009reliability}. Due to its extended downtime after a failure (which is captured in the corresponding maintenance cost), the corrective cost of a gearbox is relatively high compared to preventive maintenance costs, see, e.g.,  \cite{nilsson2007maintenance}. Based on the values reported in the aforementioned studies, we set $c_{\text{cm}}=300000$, $c_{\text{pm}}^{\text{so}}=1000$, $\mu_2=0.31$, $\mu_1=0.31$ and $p=0.6$.  In this case, the long-run rate of cost (in euros per year) in case of only corrective replacements is equal to 46500. Furthermore, motivated by the wind industry practice, we choose three different values for $\tau$, that is $\tau \in \{0.25, 0.5, 1\}$ (years). Next, we consider three different values for $c_{\text{pm}}^{\text{uso}}$, i.e. $c_{\text{pm}}^{\text{uso}} \in \{ 2000, 3000, 4000 \}$. Finally, with regard to $\lambda$, we consider four different values, i.e. $\lambda \in \{0.5, 1, 2, 4\}$. 

In Table \ref{tab:sensitivity}, we depict the long-run rate of cost for the above mentioned values under four different policies: The first policy corresponds to replacements only at USOs ($\pi_{\text{uso}}$). The second policy corresponds to replacements only at SOs ($\pi_{\text{so}}$). The third policy is the optimal policy ($\pi_{\text{opt}}$), which is derived in Theorem \ref{thm:opt1}. Note, that it is numerically easier to obtain the optimal $\tilde{t}$ by minimizing the long-run rate of cost in Theorem \ref{them:costs}, instead of the closed-form expression in Theorem \ref{thm:opt1}, as the latter requires the derivation of a root solution. The fourth policy concerns the optimal policy, but for $p=1$. This assumption is motivated from the practice, as it is oftentimes difficult to exactly determine the value of $p$ and it is typically assumed that after a maintenance the component is restored to a perfect state. This policy is denoted by $\pi'_{\text{opt}}$.

\begin{table}[!htb]
\small
\centering
\begin{tabular}{cc|cccc|cccc|cccc} 
\toprule
                                          & \multicolumn{1}{l}{}                           & \multicolumn{4}{c}{$\tau = 0.25$}                                                                                                                                                                                                                                                                    & \multicolumn{4}{c}{$\tau = 0.5$}                                                                                                                                                                                                                                                                     & \multicolumn{4}{c}{$\tau = 1$}                                                                                                                                                                                                                                                                        \\ 
\cmidrule{3-14}
$c_{\text{pm}}^{\text{uso}}$ & \multicolumn{1}{l}{$\lambda$} & \multicolumn{1}{l}{$\pi_{\text{uso}}$} & \multicolumn{1}{l}{$\pi_{\text{so}}$} & \multicolumn{1}{l}{$\pi_{\text{opt}}$} & \multicolumn{1}{l}{$\pi'_{\text{opt}}$} & \multicolumn{1}{l}{$\pi_{\text{uso}}$} & \multicolumn{1}{l}{$\pi_{\text{so}}$} & \multicolumn{1}{l}{$\pi_{\text{opt}}$} & \multicolumn{1}{l}{$\pi'_{\text{opt}}$} & \multicolumn{1}{l}{$\pi_{\text{uso}}$} & \multicolumn{1}{l}{$\pi_{\text{so}}$} & \multicolumn{1}{l}{$\pi_{\text{opt}}$} & \multicolumn{1}{l}{$\pi'_{\text{opt}}$}  \\ 
\cmidrule(r){1-14}
\multirow{4}{*}{2000}                     & 0.5                                            & 31674                                                                       & 7624                                                                       & 7193                                                                        & 7208                                                                         & 31674                                                                       & 12927                                                                      & 11627                                                                       & 11647                                                                        & 31674                                                                       & 20301                                                                      & 17134                                                                       & 17156                                                                         \\
                                          & 1                                              & 24139                                                                       & 7624                                                                       & 6815                                                                        & 6840                                                                         & 24139                                                                       & 12927                                                                      & 10583                                                                       & 10614                                                                        & 24139                                                                       & 20301                                                                      & 14855                                                                       & 14886                                                                         \\
                                          & 2                                              & 16522                                                                       & 7624                                                                       & 6183                                                                        & 6221                                                                         & 16522                                                                       & 12927                                                                      & 9007                                                                        & 9049                                                                         & 16522                                                                       & 20301                                                                      & 11794                                                                       & 11828                                                                         \\
                                          & 4                                              & 10368                                                                       & 7624                                                                       & 5258                                                                        & 5307                                                                         & 10368                                                                       & 12927                                                                      & 7023                                                                        & 7067                                                                         & 10368                                                                       & 20301                                                                      & 8469                                                                        & 8498                                                                          \\ 
\cmidrule(lr){1-14}
\multirow{4}{*}{3000}                     & 0.5                                            & 31842                                                                       & 7624                                                                       & 7230                                                                        & 7255                                                                         & 31842                                                                       & 12927                                                                      & 11687                                                                       & 11725                                                                        & 31842                                                                       & 20301                                                                      & 17224                                                                       & 17265                                                                         \\
                                          & 1                                              & 24393                                                                       & 7624                                                                       & 6883                                                                        & 6927                                                                         & 24393                                                                       & 12927                                                                      & 10691                                                                       & 10751                                                                        & 24393                                                                       & 20301                                                                      & 15010                                                                       & 15068                                                                         \\
                                          & 2                                              & 16863                                                                       & 7624                                                                       & 6304                                                                        & 6372                                                                         & 16863                                                                       & 12927                                                                      & 9188                                                                        & 9267                                                                         & 16863                                                                       & 20301                                                                      & 12034                                                                       & 12100                                                                         \\
                                          & 4                                              & 10778                                                                       & 7624                                                                       & 5456                                                                        & 5543                                                                         & 10778                                                                       & 12927                                                                      & 7294                                                                        & 7375                                                                         & 10778                                                                       & 20301                                                                      & 8800                                                                        & 8855                                                                          \\ 
\cmidrule(r){1-14}
\multirow{4}{*}{4000}                     & 0.5                                            & 32011                                                                       & 7624                                                                       & 7266                                                                        & 7299                                                                         & 32011                                                                       & 12927                                                                      & 11748                                                                       & 11800                                                                        & 32011                                                                       & 20301                                                                      & 17314                                                                       & 17374                                                                         \\
                                          & 1                                              & 24648                                                                       & 7624                                                                       & 6951                                                                        & 7009                                                                         & 24648                                                                       & 12927                                                                      & 10799                                                                       & 10883                                                                        & 24648                                                                       & 20301                                                                      & 15164                                                                       & 15248                                                                         \\
                                          & 2                                              & 17203                                                                       & 7624                                                                       & 6424                                                                        & 6513                                                                         & 17203                                                                       & 12927                                                                      & 9368                                                                        & 9479                                                                         & 17203                                                                       & 20301                                                                      & 12274                                                                       & 12368                                                                         \\
                                          & 4                                              & 11189                                                                       & 7624                                                                       & 5653                                                                        & 5766                                                                         & 11189                                                                       & 12927                                                                      & 7565                                                                        & 7677                                                                         & 11189                                                                       & 20301                                                                      & 9132                                                                        & 9208                                                                          \\
\bottomrule
\end{tabular}
\caption{Long-run rate of cost varying $\lambda$, $\tau$ and $c_{\text{pm}}^{\text{uso}}$, while keeping all other parameters fixed for four policies.}
\label{tab:sensitivity}
\end{table}

In Table \ref{tab:sensitivity}, we observe, across all instances, that incorporating planned maintenance can significantly reduce costs compared to only corrective maintenance, which can be reduced even further by adding opportunistic maintenance. Intuitively, due to the cost structure, only planned maintenance at SOs can considerably improve the long-term rate of cost when compared to performing only opportunistic maintenance at USOs. Finally, if we compare $\pi_{\text{opt}}$ with $\pi'_{\text{opt}}$ we do not, despite the low value for $p$, observe significant differences. From an operational management perspective, this clearly implies that, if decision makers do not have any knowledge about the value of $p$ and given a similar cost structure as in the gearbox case, assuming perfect maintenance will result in a long-run rate of cost that is close to optimal regardless of the true value of $p$. This will be valid as long as the preventive maintenance cost (at both opportunities) is very small in comparison to the corrective maintenance cost, as is the case of the gearbox costs. As a rule of thumb, one can easily compute the expected number of maintenances (planned or opportunistic) required for a successful preventive maintenance and based on this compute the long-run rate of preventive maintenance cost (approximately in the order of $\max\{c_{\text{pm}}^{\text{so}},c_{\text{pm}}^{\text{uso}}\}/p$) and compare it with the corrective cost. If the corrective cost is significantly higher, then one may assume that there is no significant difference between $\pi_{\text{opt}}$ and $\pi'_{\text{opt}}$, and as a consequence there is no significant difference in the values of the optimal policies under the imperfect and perfect maintenance. In the next section, we investigate the savings that can be obtained by improving the performance of a repair when a decision maker has some knowledge regarding the value of $p$. 

\subsection{Influence of imperfect maintenance}
\label{subsec:influence}
Let $\pi^{(p)}_{\text{opt}}$ represent the optimal policy as a function of the successful preventive maintenance probability $p$ and let $C(\pi^{(p)}_{\text{opt}})$ denote the long-run rate of cost when the policy is $\pi^{(p)}_{\text{opt}}$. To demonstrate the effect of $p$ in the rate of cost, we compute the relative difference in the cost of not having a perfect preventive maintenance as a function of $p$. This relative difference is denoted by  $\delta(p)$ and it is equal to $$\delta(p) = \frac{C(\pi^{p}_{\text{opt}})-C(\pi^{1}_{\text{opt}})}{C(\pi^{1}_{\text{opt}})} \cdot 100 \%.$$ 
$\delta(p)$ indicates how much extra cost is incurred due to imperfect maintenance, and thus shows the benefit of improving the probability of executing a perfect maintenance. 

In this numerical example, similarly to before we choose $\mu_2=0.31$, and $\mu_1=0.31$. 
Furthermore, we set $\lambda = 4$ and $\tau = 1$. Figure \ref{fig:graph1} shows $\delta(p)$ for $p \in [0.5,1]$ under two different cost structures (denoted by $\delta(p)^1$ and $\delta(p)^2$, respectively). Figure \ref{fig:graph2} depicts the corresponding optimal values for $\tilde{t}$ for both cost structures, denoted by $t^1$ and $t^2$, respectively. We use the same cost structure as in the previous section, i.e. for $\delta(p)^1$, we consider $c_{\text{pm}}^{\text{so}} = 1000, c_{\text{pm}}^{\text{uso}}=2000$ and $c_{\text{cm}} = 300000$, whereas, for $\delta(p)^2$, we consider  $c_{\text{pm}}^{\text{so}} = 26500, c_{\text{pm}}^{\text{uso}}=28800$ and $c_{\text{cm}} = 75500$. The choice for the preventive maintenance cost at SOs and USOs in the second cost structure is common in the lithography industry (see \cite{zhu2017condition}). Based on Figure \ref{fig:graph1}, we can conclude that,  under both cost structures, significant costs can be saved by improving the probability of executing a perfect preventive maintenance (e.g., by training).

\begin{figure}[!htb]
\centering
  \includegraphics[width=0.7\linewidth]{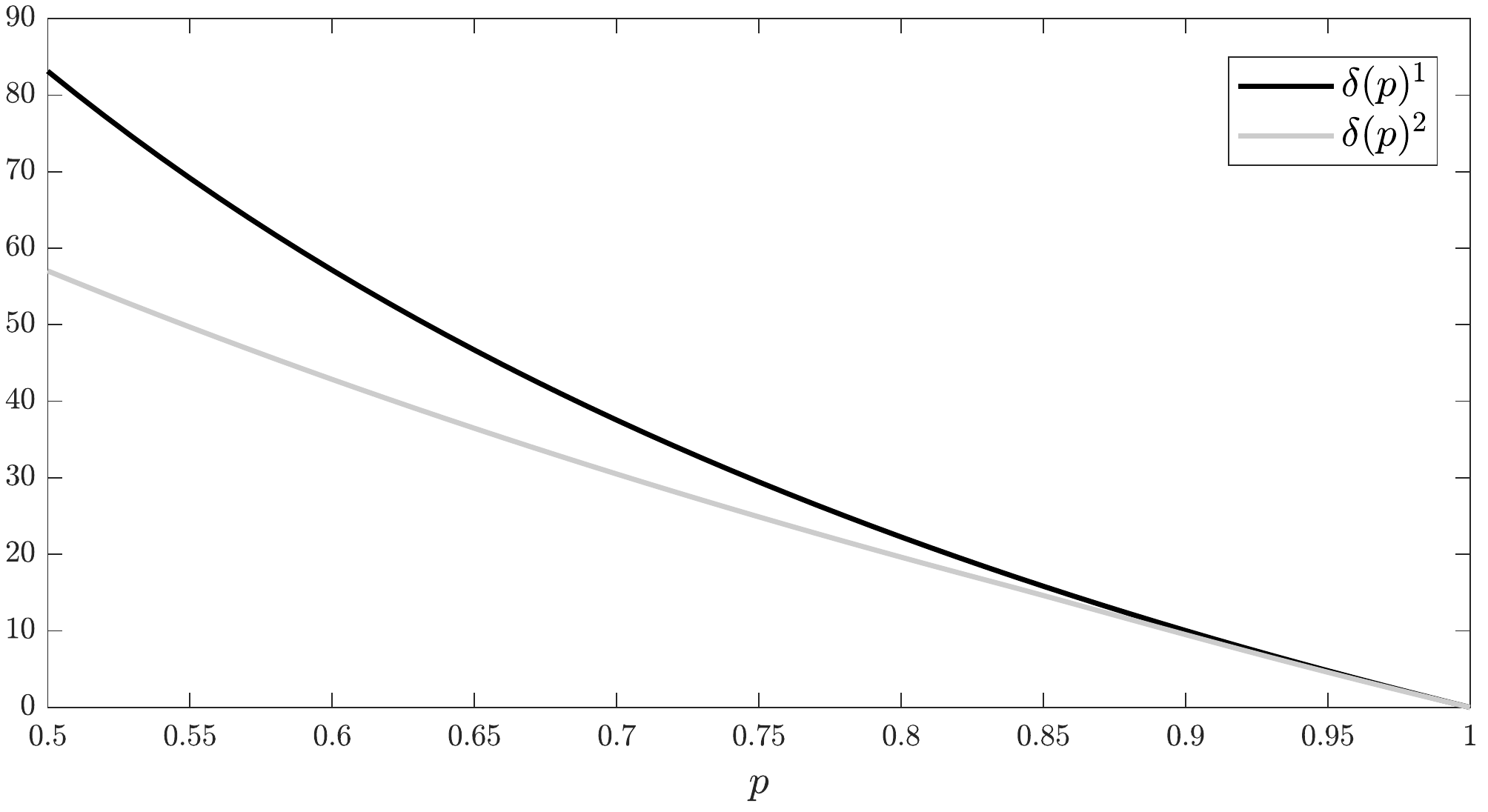}
  \caption{$\delta(p)^1$ and $\delta(p)^2$  for $ p \in [0.5,1]$ with $c_{\text{pm}}^{\text{so}} = 1000, c_{\text{pm}}^{\text{uso}}=2000$ and $c_{\text{cm}} = 300000$ for $\delta(p)^1$, and $c_{\text{pm}}^{\text{so}} = 26500, c_{\text{pm}}^{\text{uso}}=28800$ and $c_{\text{cm}} = 75500$ for $\delta(p)^2$.}
  \label{fig:graph1}
\end{figure}

The optimal policy $(\tilde{t})$, denoted by $t^{1}$ and $t^{2}$, under the first and second cost structure, respectively, is equal to $t^{1}\approx 0.08$ and $t^{2}\approx 0.39$ in case of perfect repairs. In Figure \ref{fig:graph2}, where we plot  $t^{1}$ and $t^{2}$ as a function of $p$, we observe the following regarding the influence of $p$ on the maintenance planning: If the preventive maintenance cost (at both opportunities) is very small compared to the cost of corrective maintenance, the order of the total preventive maintenance cost incurred until a successful preventive maintenance compared to the corrective maintenance cost is still maintained. Therefore, the maintenance planning does not alter that much regardless of the value of $p$, where the optimal policy is to almost always perform preventive maintenance at USOs for all values of $p\in [0.5,1]$. This also explains the small discrepancy between $\pi_{\text{opt}}$ and $\pi'_{\text{opt}}$ in Table \ref{tab:sensitivity}. This is different in the case of the second cost structure, where the maintenance planning changes substantially as a function of $p$. Whereas in the perfect case, the optimal policy is to perform preventive maintenance at a USO if the residual time until the next SO is larger than 0.39, for $p \lessapprox  0.83$, it is optimal to never perform preventive maintenance at a USO. Here, the order of the total preventive maintenance cost incurred until a successful preventive maintenance compared to the corrective maintenance cost is not maintained.

\begin{figure}[!htb]
\centering
  \includegraphics[width=0.7\linewidth]{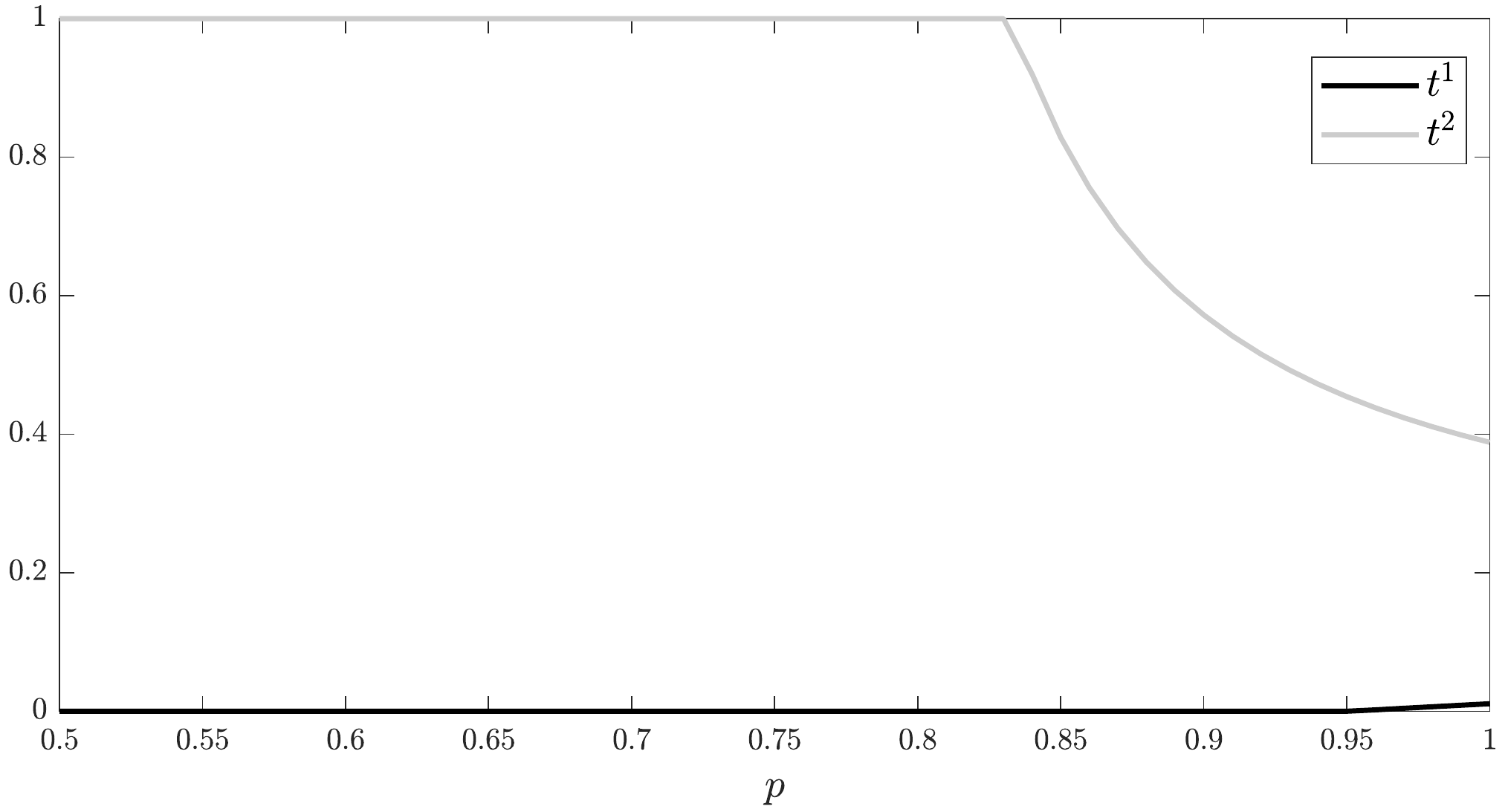}
  \caption{$t^{1}$ and ${t}^{2}$ for $ p \in [0.5,1]$ with $c_{\text{pm}}^{\text{so}} = 1000, c_{\text{pm}}^{\text{uso}}=2000$ and $c_{\text{cm}} = 300000$ for $t^1$, and $c_{\text{pm}}^{\text{so}} = 26500, c_{\text{pm}}^{\text{uso}}=28800$ and $c_{\text{cm}} = 75500$ for $t^2$.}
  \label{fig:graph2}
\end{figure}

Also in the opposite cost structure, i.e. $c_{\text{pm}}^{\text{uso}}<c_{\text{pm}}^{\text{so}}$ (similar examples can be found for $c_{\text{pm}}^{\text{uso}}=c_{\text{pm}}^{\text{so}}$), the maintenance planning can be influenced significantly by the imperfect repair probability. For instance, consider the setting with $\mu_1=1.1, \mu_2 =0.9$,  $c_{\text{pm}}^{\text{so}} = 4500$, $c_{\text{pm}}^{\text{uso}}=4000$, $c_{\text{cm}}=10000$, and $\lambda=0.5$. In  case of perfect repairs (i.e. $p=1$), the optimal policy is to perform preventive maintenance in state 1 at both SOs and USOs, and to do nothing otherwise (cf. Theorem \ref{prop:unequal}). However, if  $0.72 \lessapprox p \lessapprox 0.83$, the optimal policy is to only perform preventive maintenance at USOs and if $p \lessapprox 0.72$, then the optimal policy is to never perform PM. This example illustrates the influence of the imperfect repair probability on the maintenance planning. 
 
\subsection{Deferring of planned maintenance}
\label{sec:Num_Def_M}
In this section, we illustrate the change introduced by the action of deferring planned maintenance after the occurrence of a successful maintenance in three numerical examples that relate to the wind industry, the lithography industry, and to an artificially created example.

Figure \ref{fig:wind} shows the long-run rate of cost for both the deferral and no deferral case for the example with data stemming from the wind industry. Again, with regard to the cost parameters, we used $c_{\text{pm}}^{\text{so}} = 1000, c_{\text{pm}}^{\text{uso}}=2000$ and $c_{\text{cm}} = 300000$. With regard to the other parameters, we set $\lambda = 4$, $\tau = 1$, $\mu_1 = 0.31$, $\mu_2 = 0.31$ and $p=0.6$. We can observe that deferring the planned maintenance both significantly increases the long-run rate of cost under the optimal policy (an increase of 28.14\% from 8468.87 to 10852.15) and changes the value connected to the optimal policy, $\tilde{t}$ from 0.112 to 0.  

\begin{figure}[!ht]
\centering
  \includegraphics[width=0.5\linewidth]{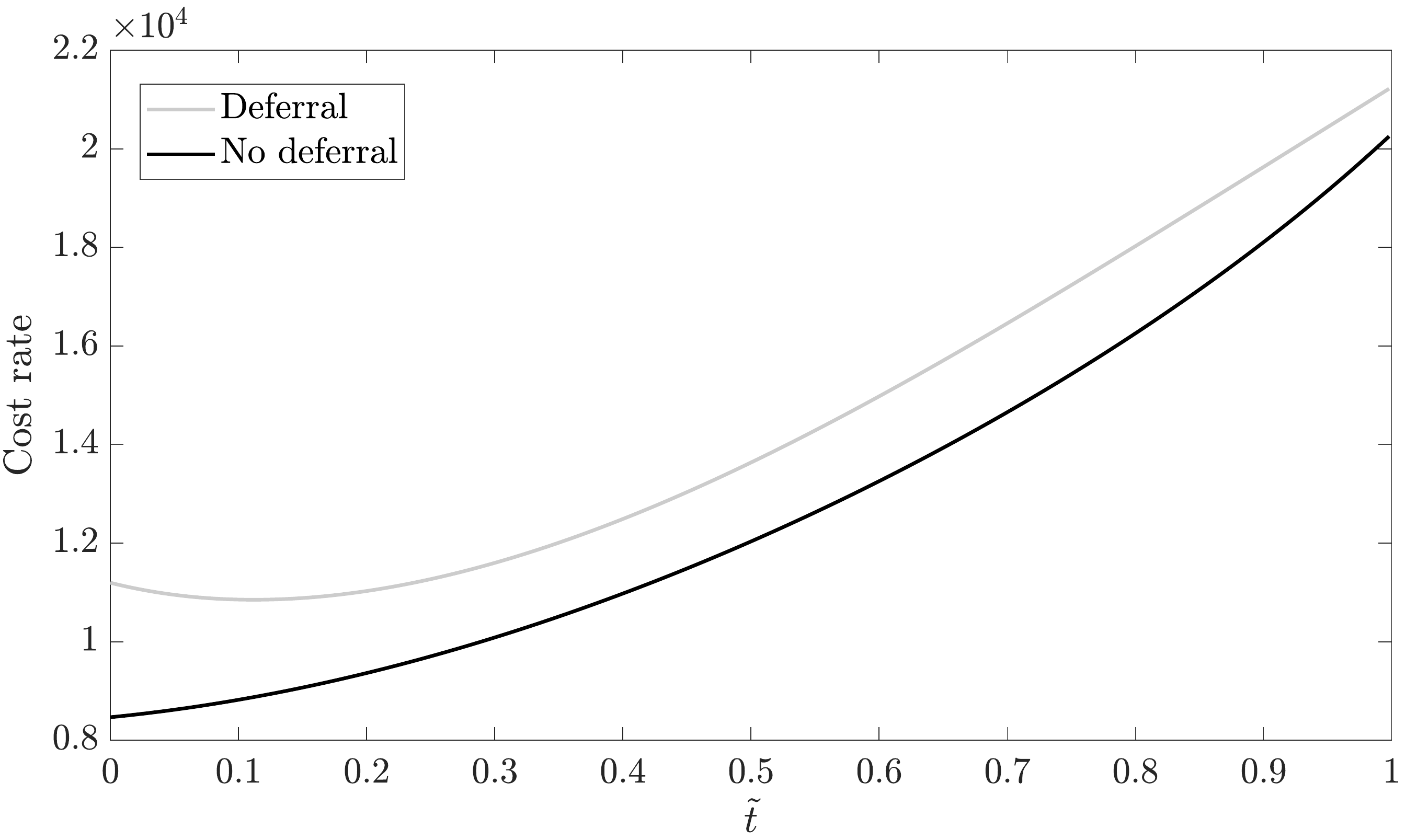}
  \caption{Cost rate in case of deferral and of no deferral for wind industry example. Optimal $\tilde{t}$ is equal to 0.112 and 0 for deferral and no deferral, respectively.}
  \label{fig:wind}
\end{figure}

Figure \ref{fig:lithDef} and Figure \ref{fig:LithNoDef} depict the long-run rate of cost for both the deferral and the no deferral case, respectively, based on the values of the lithography industry example. We use the  same cost parameters as in Section \ref{subsec:influence}, that is $c_{\text{pm}}^{\text{so}} = 26500, c_{\text{pm}}^{\text{uso}}=28800$ and $c_{\text{cm}} = 75500$. The other parameters remain unchanged, i.e. $\lambda = 4$, $\tau = 1$, $\mu_1 = 0.31$, $\mu_2 = 0.31$ and $p=0.6$. Again, we observe the same influence of deferring the planned maintenance on both the long-run rate of cost under the optimal policy  (an increase of 6533.3 \% from 12840.12 to 851727.53) and on the value of $\tilde{t}$ associated with the optimal policy (from 1 to 0.175) , similarly to the numerical example for the wind industry. The drastic increase is due to the cost structure, and more explicitly, it is due to the preventive maintenance costs values (both at scheduled and unscheduled opportunities), which are relatively much closer to the corrective maintenance cost in comparison to the wind industry example.  

\begin{figure}[!ht]
    \centering
    \begin{subfigure}[b]{0.46\textwidth}
\centering
  \includegraphics[width=1\linewidth]{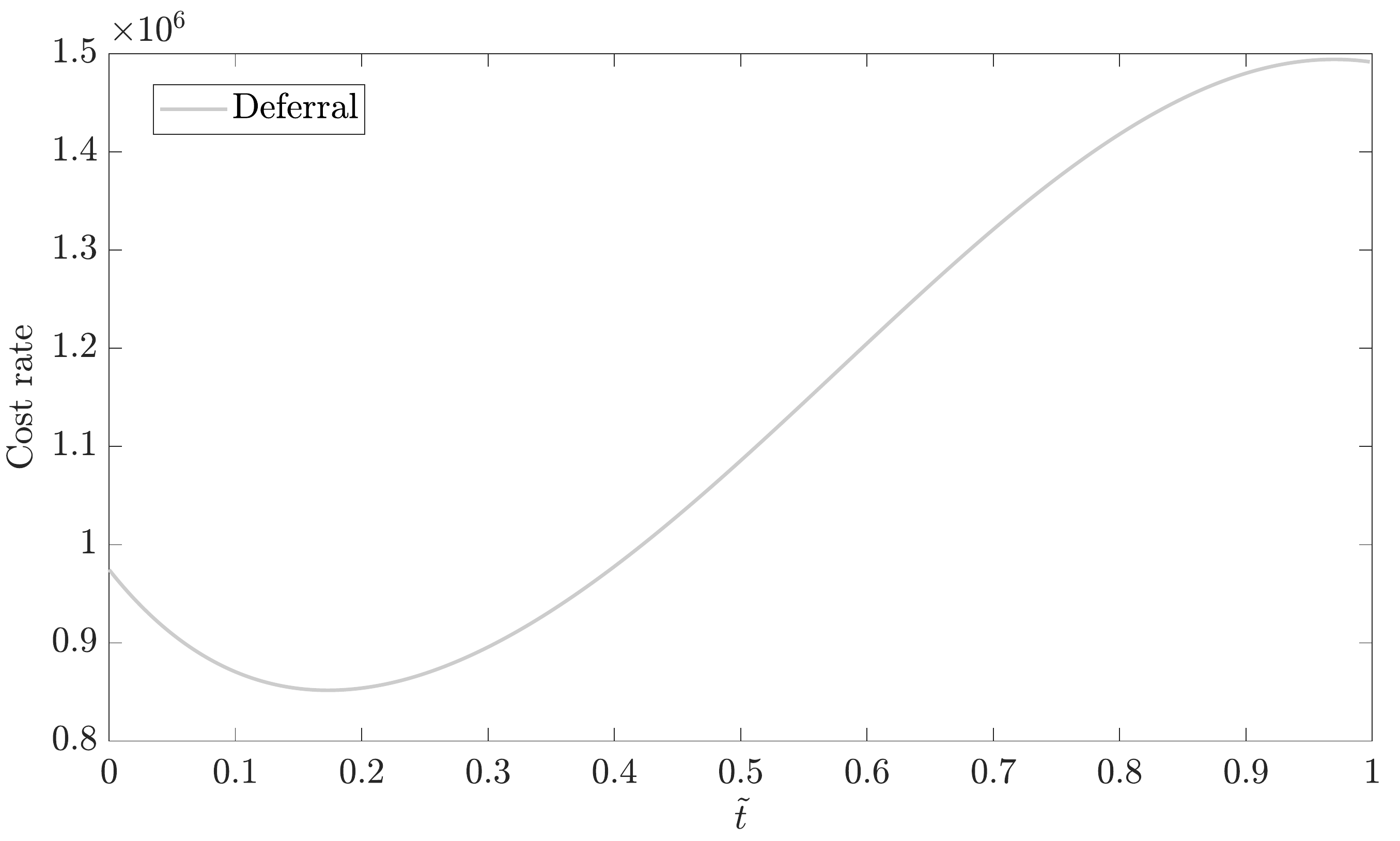}
  \caption{Cost rate in case of deferral for lithography industry example. Optimal $\tilde{t}$ is equal to 0.175.}
  \label{fig:lithDef}
    \end{subfigure}
    ~ 
    ~ 
    \begin{subfigure}[b]{0.46\textwidth}
\centering
  \includegraphics[width=1\linewidth]{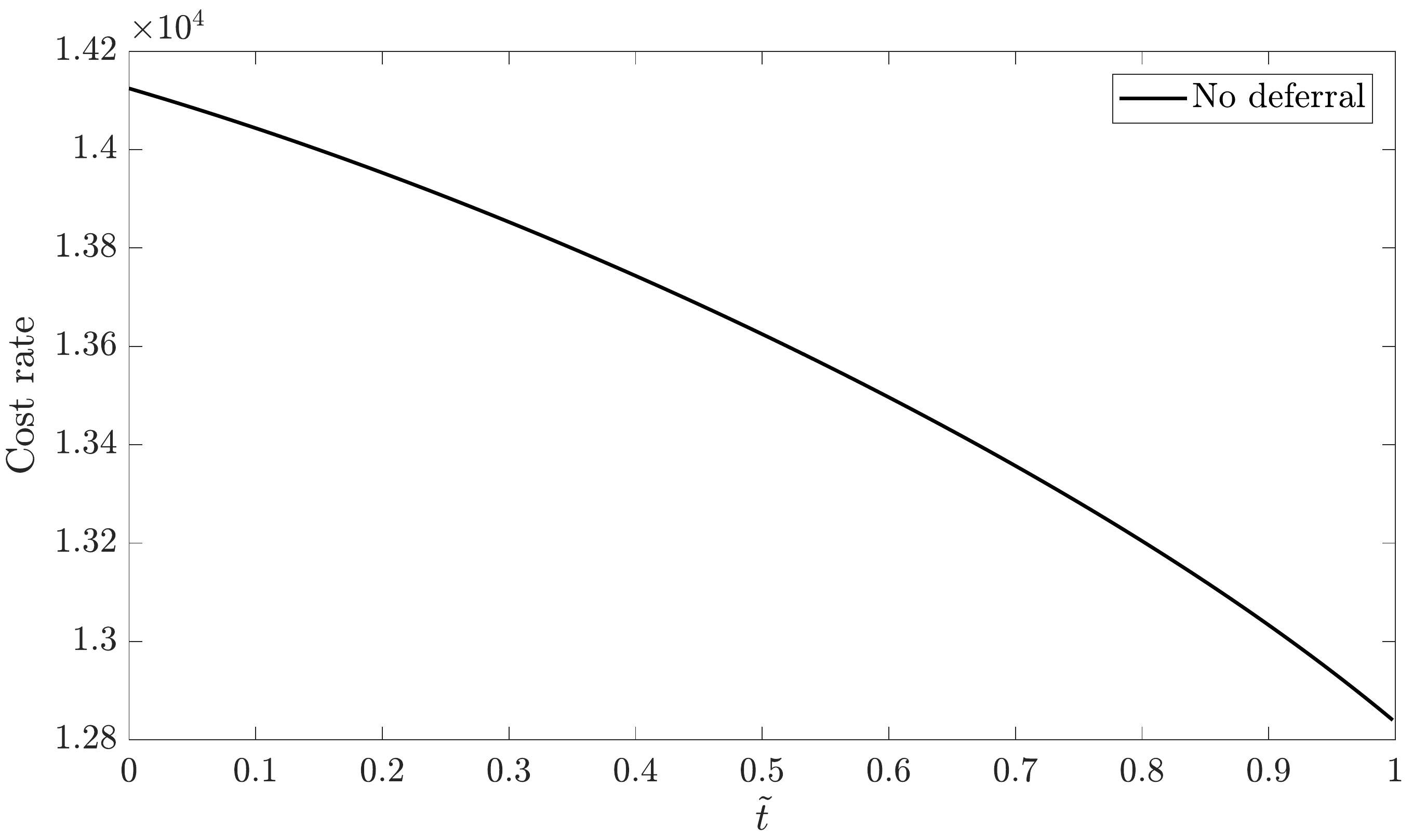}
  \caption{Cost rate in case of no deferral for lithography industry example. Optimal $\tilde{t}$ is equal to 1.}
  \label{fig:LithNoDef}
    \end{subfigure}
    \caption{Cost rate for lithography industry example.}\label{fig:lith}
\end{figure}

To illustrate that the opposite effect (albeit to a much lesser degree than in the previous two examples) can also hold, we create an artificial example where we set $c_{\text{pm}}^{\text{so}} = 5000, c_{\text{pm}}^{\text{uso}}=10000$ and $c_{\text{cm}} = 19000$, and $\lambda = 4$,$\tau = 4$,$\mu_1 = 1$, $\mu_2 = 0.4$ and $p=0.5$. Figure \ref{fig:toy} depicts the long-run rate of cost for both the deferral and the no deferral case for this example. Here we observe that for all values of $\tilde{t}$, cost savings can be obtained by deferring planned maintenance after the occurrence of a successful opportunistic maintenance. More specifically, whereas the optimal value of $\tilde{t}$ is equal to 1 for both cases, the long-run rate of cost under the optimal policy decreases with 0.88\% from 6458.97 to 6402.44, when deferring planned maintenance.

\begin{figure}[!ht]
\centering
  \includegraphics[width=0.5\linewidth]{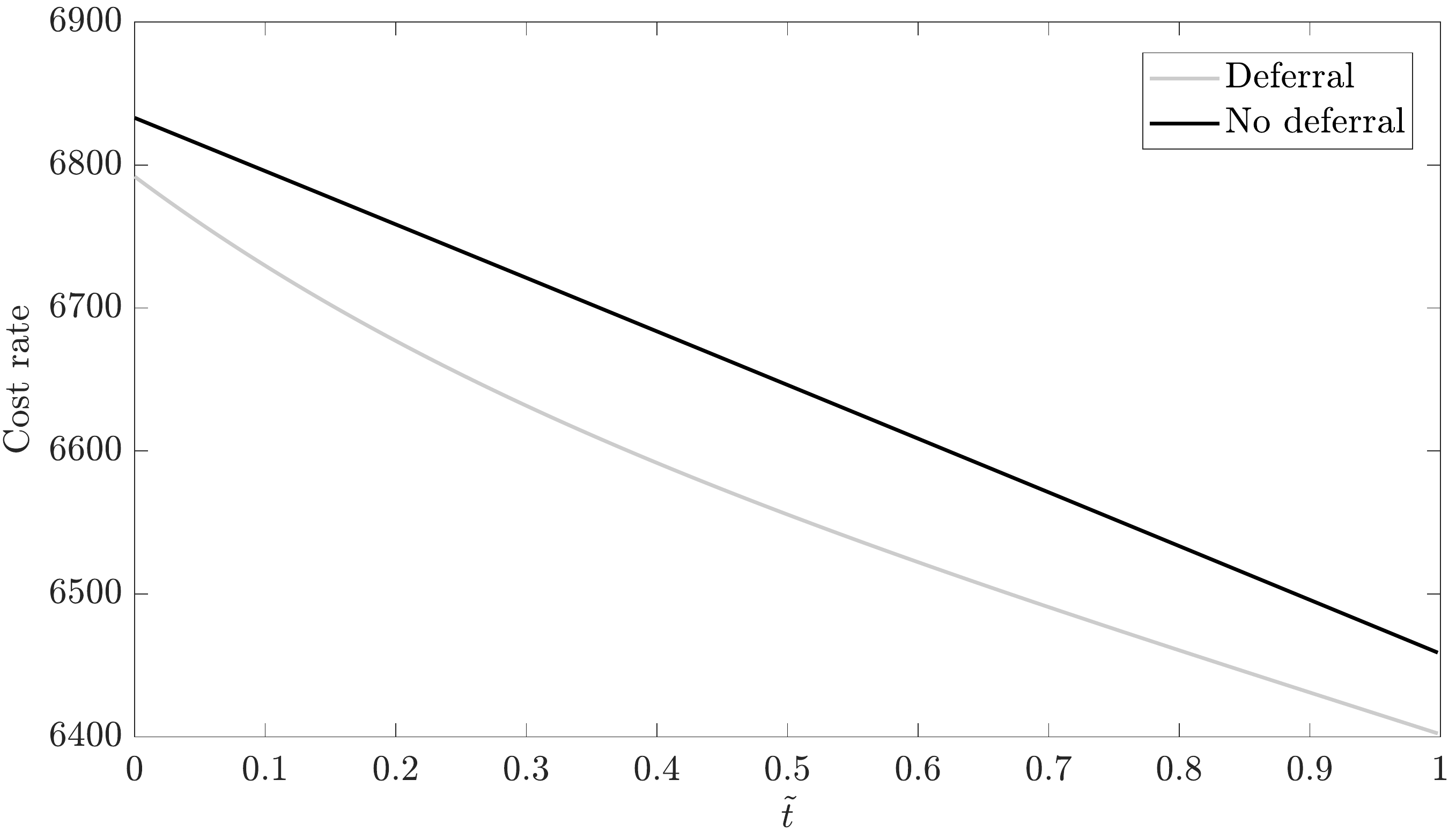}
  \caption{Cost rate in case of deferral and no deferral for artificial example. Optimal $\tilde{t}$ is equal to 1 for both deferral and no deferral.}
  \label{fig:toy}
\end{figure}

\section{Conclusion}
\label{sec:conclusion}
In this paper, we considered the maintenance policy for a 3-state component degrading over time with corrective replacements at failures and preventive replacements at both scheduled and unscheduled opportunities under imperfect repair. By formulating this problem as a semi-Markov decision process, we were able to characterize the structure of the optimal maintenance policy as a control limit policy, where the control limit depends on the time until the next planned maintenance opportunity. Using this approach, a closed-form expression for the optimal control limit was derived. Within this class of control limit policies, we derived, using the theory of regenerative processes, an explicit expression for the long-run rate of cost. Using a similar approach based on renewal theory, we derived an expression for the long-run rate of cost in the case when planned maintenance is deferred after the occurrence of a successful opportunistic maintenance.

A cost comparison with other suboptimal policies has been examined, which illustrated the benefits of optimizing the maintenance policy. Specifically, it was found that incorporating planned maintenance can significantly reduce costs compared to only corrective maintenance, which can be reduced even further by adding opportunistic maintenance. Moreover, numerical results indicate that the extent of the impact of the perfect repair probability on the optimal policy depends on the underlying cost structure. It was also shown that substantial cost savings can be obtained by improving the perfect repair probability. Finally, our numerical examples indicate that the deferral of planned maintenance after the occurrence of a successful opportunistic maintenance may impact the total cost in both a negative and positive way.

There are a number of extensions and topics for future research. The most important direction is to consider the network dependency on the level of the structural degradation and failure dependencies, i.e. to consider a multi-dimensional process that captures the degradation of the various assets in the network. Such a future direction would be particularly interesting in the case of a small number of assets for which  the Poisson approximation for the opportunistic maintenance may not be accurate. In addition, another very interesting research direction would be to consider a more general model in which the condition of the system degrades through $N>2$ states. Next, in this analysis, we have assumed that the condition of the system is fully observable. However, in many real applications, condition monitoring data such as spectrometric oil data or vibration data gives only partial information about the underlying state of the system. From this perspective, it would be interesting to extend the model at hand to a partially observable model in which the condition monitoring data are stochastically related to the true system state. Finally, the results in this paper are valid for systems with hypo-exponentially distributed lifetimes. Future research could relax this assumption by considering a phase-type  lifetime distribution. 

\section*{Acknowledgments}
\noindent The authors gratefully acknowledge the contribution of S. Kalosi in the early stages of the preparation of the work. The authors would like to thank M. Barbieri, J. Korst, and V. Pronk (all Philips Research), and O. J. Boxma and G. J. van Houtum (both Eindhoven University of Technology) for their time and advice in the preparation of this work. 
The work of C. Drent is supported by the Data Science Flagship framework, a cooperation between the Eindhoven University of Technology and Philips. The work of S. Kapodistria is supported by the NWO Gravitation Project ``NETWORKS'' of the Dutch government.

\appendix
\section{Optimality equations for semi-Markov decision process}
\label{sec:appendixmdp}
We consider the so-called ratio-average cost for a controlled semi-Markov decision process, which corresponds to the limes superior of the expected total cost over a finite number of jumps divided by the expected cumulative time of these jumps, see \cite{ross1970average,Feinberg_1994,Schal_1992}, for instance.

We shall use  here the definition of a  controlled semi-Markov decision process  from \cite{lippman1975dynamic,yushkevich1982semi,Jaskiewicz2004}. A controlled semi-Markov decision process  is specified by five objects: a Borel state space $\mathcal{S}$, a Borel
action space $\mathcal{A}$, a law of motion -- a measurable projection determining the state as a function of an action, a transition function (transition law) $\mathcal{P}$ -- a probability measure depending measurably on the state and the action, and a  reward (or cost) function $c $. 

The process is observed at time $t = 0$ to be in some state $x_0\in\mathcal{S}$. At that time an action $a_0\in\mathcal{A}_{x_0}$ is chosen, where $\mathcal{A}_{x_0}$ is a compact set of actions available in state $x_0$. The set of all actions is $\mathcal{A}$ and is also assumed to be a Borel state space.

For the problem at hand, the state space is 
$$\mathcal{S}=\left\{(i,j,t):\ i\in \{1,2\}, \ j\in\{\text{SC},\text{USO}\},\ t\in(0,\tau)\right\}\cup \left\{(i,\text{SO},0):\ i\in \{1,2\}\}\right\}$$ and the action space is $\mathcal{A} =\{\text{perform PM, do nothing, perform CM}\}$,  cf. Section \ref{Sec:AverageCostCriterion}. 

If the current state is $x_0$ and action $a_0$ is selected, then the immediate cost $c(x_0;a_0)$ is
incurred, and the system remains in state $x_0 $ for a random time $t_1$, with the cumulative distribution depending only on $x_0$ and $a_0$. Afterward, the system jumps to the state $x_1 $ according to the probability measure (transition law) $\mathcal{P}( \cdot\,|\, x_0, a_0,t_1)$. This procedure yields a trajectory $(x_0,a_0,t_1,x_1,a_1,t_2,\ldots)$ of some stochastic process, where $x_n$ is the state, $a_n$ is the control variable and $t_n$ is the time of the $n$-th transition, $n=0,1,\ldots$. In the sequel, we shall refer to the corresponding random variables by means of their capital letters: $T_n$ -- the random time of $n$-th transition, for $n=1,2,\ldots$ with $T_0:=0$, $X_n$ -- the state at time $T_n$, and $A_n$ -- the action at time $T_n$.

Let $H_n$ be the space of admissible histories up to the $n$-th transition, $H_n:=(\mathcal{S}\times\mathcal{A}\times [0,\infty))^n\times \mathcal{S}$ and $H_0:=\mathcal{S}$. An element $h_n$ of $H_n$ is called a partial history of the process and is of the form $h_n=(x_0,a_0, t_1, \ldots, x_{n-1},a_{n-1}, t_n, x_n)$. A control policy (or policy) is a sequence $\{\pi_n\}$, where each $\pi_n$ is a conditional probability $\pi_n(\cdot\,|\, h_n)$ on the control set $\mathcal{A}_{x_n}$, given the entire history $h_n$ such that $\pi_n(\mathcal{A}_{x_n}\,|\, h_n)=1$, for all $h_n$, $n=1,2,\ldots$.
The class of all policies is denoted by $\Pi$ and let $\Pi_{DS}$ denote the class of all deterministic stationary policies.

For each initial state $x_0\in\mathcal{S}$ and for each policy $\pi\in\Pi$, there exists a  unique probability measure $\mathbb{P}^\pi_{x_0}$ such that 
\begin{align*}
\mathbb{P}^\pi_{x_0}[A_n\in A\,|\, h_n]&=\pi_n[A\,|\, h_n],\text{ for a Borel set }A\subset\mathcal{A},\\
\mathbb{P}^\pi_{x_0}[T_{n+1}-T_{n}\in S, X_{n+1}\in X\,|\, h_n,a_n]&=\mathcal{P}_{x_n}^{a_n}[S,X],\text{ for Borel sets }X\subset\mathcal{S} \text{ and } S\subset\mathbb{R},\\
\mathbb{P}^\pi_{x_0}[X_{n+1}\in X\,|\, h_n,a_n, T_{n+1}-T_n=s]&=\mathcal{P}_{x_n}^{a_n}[X\, |\, s],\text{ for a Borel set }X\subset\mathcal{S},\\
\mathbb{P}^\pi_{x_0}[T_{n+1}-T_{n}\leq s \,|\, h_n,a_n]&=F^{a_n}_{x_n}(s), s\in\mathbb{R}.
\end{align*}
Further, let $\tau(x,a)$ denote the conditional mean sojourn (holding) time spent in state $x$ under action $a$, i.e.
\[\tau(x;a):=\int_{0}^\infty s\mathrm{d} F_x^a(s),\]
and let 
$\tilde{F}_{x}^a(\alpha)$ denote the Laplace-Stieltjes transform of the sojourn time spent in state $x$ under action $a$, i.e.
\[\tilde{F}_{x}^a(\alpha):=\int_{0}^\infty e^{-\alpha s}\mathrm{d} F_x^a(s).\]

For the problem at hand, the cost function is defined as follows
\begin{align*}
c(x;a)=\begin{cases}
0,&\text{ if } x\in\mathcal{S}, a=\{\text{do nothing}\},\\
c_{\text{cm}},&\text{ if } x=(1,\text{SC},t),t\in(0,\tau),a=\{\text{perform CM}\},\\
c_{\text{pm}}^{\text{uso}},&\text{ if } x=(i,\text{USO},t), i=1,2,t\in(0,\tau), a=\{\text{perform PM}\},\\
c_{\text{pm}}^{\text{so}},&\text{ if } x=(i,\text{SO},0), i=1,2,a=\{\text{perform PM}\}.\\
\end{cases}
\end{align*}

Let $\mathcal{P}_{x_n}^{a_n}(s,x_{n+1})$ denote the joint density/mass distribution of the  transition time  $T_{n+1}-T_n$ and the allowed next state $X_{n+1}$, given the current state $X_n=x_n$ and the allowed action $a_n$. For $x_n=(1,\text{SC},t)$, $t\in(0,\tau)$, and $a_n=\{\text{perform CM}\}$,
\begin{align*}
\mathcal{P}_{x_n}^{a_n}\left(s,(2,\text{SC},t-s)\right)
&=\frac{\mu_2}{\lambda+\mu_2}(\lambda+\mu_2)e^{-(\lambda+\mu_2)s}=\mu_2\, e^{-(\lambda+\mu_2)s},\, s\in[0,t)\\
\mathcal{P}_{x_n}^{a_n}\left(s,(2,\text{USO},t-s)\right)
&= \lambda \, e^{-(\lambda+\mu_2)s},\, s\in[0,t)\\
\mathcal{P}_{x_n}^{a_n}\left(t,(2,\text{SO},0)\right)
&=e^{-(\lambda+\mu_2)t}.
\end{align*}
For the derivation of the above probabilities, it suffices to note that there are three possible evolutions in terms of the state of the system: either an SO or an SC or  a USO, where the time till the  SO is equal to $t$, while the times till the next SC and the USO are exponentially distributed with rates $\mu_2$ and $\lambda$, respectively. The probabilities for $x_n=(2,\text{USO},t)$  and $a_n=\{\text{do nothing}\}$  or $a_n=\{\text{perform PM}\}$ are identical. 
The remaining probabilities are obtained using very similar arguments. 

From the joint distributions, the marginal cumulative distribution of the transition time $T_{n+1}-T_n$ can be immediately derived as follows, for $x_n =(1,\text{SC},t)$ and $a_n=\{\text{perform CM}\}$,
\begin{align*}
F_{x_n}^{a_n}(s)&=1-e^{-(\lambda+\mu_2)s},\, s\in(0,t),\\
F_{x_n}^{a_n}(s)&=1,\, s\geq t.
\end{align*} 
The distribution of the transition time from state $x_n=(2,\text{USO},t)$  under actions $a_n=\{\text{do nothing}\}$  or $a_n=\{\text{perform PM}\}$ are identical. 
The rest of the marginal cumulative distributions for the other states and actions follow analogously.

Having fully defined the probabilities for the problem at hand, we proceed in providing, following the proofs in \cite{Bhattacharya}, the proposition below that guarantees that (1) a dynamic programming equation holds for the optimal reward  (this equation is typically referred to as the average optimality equality or as the Bellman equation), and (2)  a deterministic stationary policy (optimal for long-run average reward) is provided by this equation.
\begin{proposition}\label{prop:ACOI-full}
For the model at hand, there exist a bounded function  $V(\cdot)$ and a constant $g$ such that
\begin{align}
V(x)&=\min_{a\in \mathcal{A}_x} \Big \{ c(x;a)+\int_yV(y)\mathcal{P}_x^a(\mathrm{d}y)-g \,\tau(x;a) \Big\}, \, \forall x\in \mathcal{S}.\label{Eq:Bellman_Gen}
\end{align}
Moreover, the deterministic stationary policy $\pi^{*(\infty)}\in\Pi_{DS}$ is optimal for the ratio-average cost criterion with 
\begin{align*}
g=\inf_{\pi\in\Pi_{DS}} J(x,\pi): = J^*(x)
\end{align*}
where
\begin{align}
J(x,\pi):=\lim\sup_{n\to\infty} \frac{\mathbb{E}_x^\pi \left[\sum_{k=0}^{n-1}c(X_{k},A_{k})\right]}{\mathbb{E}_x^\pi \left[T_n\right]}\equiv \lim\sup_{n\to\infty} \frac{\mathbb{E}_x^\pi \left[\sum_{k=0}^{n-1}c(X_{k},A_{k})\right]}{\mathbb{E}_x^\pi \left[\sum_{k=0}^{n-1}\tau(X_{k},A_{k})\right]}, \, \pi\in \Pi_{DS}.\label{Eq:AverageCostEq}
\end{align}
\end{proposition}
\begin{proof}
The proof of the proposition relies on the fact that the costs $c(x;a)$ are non-negative and upper bounded by $c_{\text{cm}}$. We follow here the ideas presented in \cite{Bhattacharya} and in Theorems 10.3.1 \& 10.3.6 in \cite[Sections 10.4 and 10.5]{Hernandez_2012}. Following the ideas therein, we consider the corresponding $\alpha$-discounted cost criterion 
\begin{align*}
V_\alpha(x,\pi)=\mathbb{E}_x^\pi \left[\sum_{k=0}^{\infty}e^{-\alpha T_k}c(X_{k},A_{k})\right]
\end{align*}
and $V_\alpha(x)=\inf_{\pi\in\Pi}V_\alpha(x,\pi)$. The main steps in the proof of the proposition are 
\begin{description}

\item[Step 1:] Show that the optimal reward $V_\alpha (x)$ under discounting is
continuous and bounded. The latter follows easily by noting that $V_\alpha (x)$ is bounded by $V_\alpha(\pi_{\text{DN}},x)$, where $\pi_{\text{DN}}$ denotes the policy of doing nothing at all opportunities, unless the component fails, in which case it is mandatory to do corrective maintenance. This yields
\begin{align*}
V_\alpha(x)\leq c_{\text{cm}}\frac{\frac{\mu_1}{\mu_1+\alpha}}{1-\frac{\mu_1}{\mu_1+\alpha}\frac{\mu_2}{\mu_2+\alpha}},\ \forall x\in\mathcal{S}. 
\end{align*}
Analogously, 
$$g\equiv J^*(x)\leq c_{\text{cm}}\frac{\mu_1\mu_2}{\mu_2+\mu_1}.$$
See Appendix \ref{Proof of Lemma 10.4.1} for further details.

\item[Step 2:] Show that the discounted Bellman equation  
\begin{align}\label{Eq:Bellman_Gen_Discounted}
V_\alpha(x)&=\min_{a\in \mathcal{A}_x} \Big \{ c(x;a)+\int_{s} \int_{y} e^{-\alpha s}V(y)\mathcal{P}^a_x (\mathrm{d}y\,|\,s)\,\mathrm{d}F_x^a(s)\Big\}, x\in \mathcal{S},
\end{align}
holds. Also,
there exists a Borel measurable function that minimizes the right side of the discounted Bellman equation for every $x\in\mathcal{S}$. The deterministic stationary policy is optimal under discounting. The proof follows verbatim the steps in \cite[Theorem 3.1 on page 227]{Bhattacharya}.

\item[Step 3:] Choose $z\in\mathcal{S}$, then for all  $x\in\mathcal{S}$, show that $|V_\alpha(x)-V_\alpha(z)|$ is bounded for all $\alpha>0$. This follows oftentimes by the geometric ergodicity of the underlying Markov controlled model. In the case under consideration, this is proven by noting that from all states $x=(i,j,t) \in\mathcal{S}$, after time $t$ the system is in an SO state with probability 1. This yields
\begin{align*}
|V_\alpha(x)-V_\alpha(z)|&\leq
 c_{\text{cm}}\left(2+\left(\lambda +\mu_1+\mu_2+1+\frac{\mu_1\mu_2}{\mu_1+\mu_2}\right)(\tau_{x,1}+\tau_{z,1})\right), 
\end{align*}
with  $\tau_{x,1}<\infty$ denoting the expectation of the first passage time from state $x\in\mathcal{S}$ to state $(1,\text{SO},0)$. 
See Appendix \ref{Proof of Lemma 10.4.2} for further details. A consequence of the above finding is that, for all deterministic stationary policies $\pi\in\Pi_{DS}$, the expected average cost in  \eqref{Eq:AverageCostEq} is independent of $x$.

\item[Step 4:] Show that there exists a solution say $g$ to the average optimality equality \eqref{Eq:Bellman_Gen}. There exists a Borel measurable function $\pi^ *$ on $\mathcal{S}$ into $\mathcal{A}$ such that the maximum on the right side of \eqref{Eq:Bellman_Gen} is attained at $\pi^*(x)$, $x\in\mathcal{S}$. 
The proof follows verbatim the steps in \cite[Theorem 3.2 (a) \& (b) on page 228]{Bhattacharya}.

\item[Step 5:] Show that the stationary policy $\pi^{*(\infty)}$ is optimal for the long-run average reward and $g$ is
the optimal reward, with $g=\limsup_{\alpha\to0^+}\alpha V_\alpha(x)$.
See Appendix \ref{Proof of Lemma 10.4.3} for further details. 
\end{description}
\end{proof}

Equivalent propositions (based on different methods, but more importantly based on different assumptions regarding the geometric ergodicity) can be found for example in \cite{Jaskiewicz2001,Vega_2000,Jaskiewicz2004}. 

\subsection{Proof of Step 1} \label{Proof of Lemma 10.4.1}
Under the policy of doing nothing at all opportunities, unless the component fails in which case it is mandatory to do corrective maintenance, say $\pi_{\text{DN}}$, $V_\alpha(\pi_{\text{DN}},x)$ can be computed using first step analysis. Note that under this policy, it is not required to keep track of the remaining time to the next SO opportunity. Say $x=(i,j,\cdot)$.  If $i=1$, then after an exponentially distributed time with rate $\mu_1$, say $T_{\mu_1}$, the component will fail and a cost $c_{\text{cm}}$ will be incurred. If  $i=2$, then after a Hypo-exponentially distributed time with rates $(\mu_2,\mu_1)$, say $T_{\mu_1}+T_{\mu_2}$ (the two random times are independent), the component will fail and a cost $c_{\text{cm}}$ will be incurred. All in all, 
\begin{align}
V_\alpha(x,\pi_{\text{DN}})&=\mathbb{E}[e^{-\alpha(T_{\mu_1}+T_{\mu_2}\i_{\{i=2\}})}]\left(c_{\text{cm}}+\mathbb{E}[V_\alpha((1,\text{SC},\cdot),\pi_{\text{DN}})]\right).\label{Eq:V_Bound1}
\end{align}
Similarly,
\begin{align*}
V_\alpha((1,\text{SC},\cdot),\pi_{\text{DN}})&=\mathbb{E}[e^{-\alpha(T_{\mu_1}+T_{\mu_2})}]\left(c_{\text{cm}}+\mathbb{E}[V_\alpha((1,\text{SC},\cdot),\pi_{\text{DN}})]\right),
\end{align*}
which yields upon solving for $V_\alpha((1,\text{SC},\cdot),\pi_{\text{DN}})$ and substituting that $\mathbb{E}[e^{-\alpha T_{\mu_i}}]=\frac{\mu_i}{\mu_i+\alpha}$, $i=1,2$,
\begin{align*}
V_\alpha((1,\text{SC},\cdot),\pi_{\text{DN}})&=c_{\text{cm}}\frac{\frac{\mu_1}{\mu_1+\alpha}\frac{\mu_2}{\mu_2+\alpha}}{1-\frac{\mu_1}{\mu_1+\alpha}\frac{\mu_2}{\mu_2+\alpha}}. 
\end{align*}
Combining the last equation with \eqref{Eq:V_Bound1} yields
\begin{align*}
V_\alpha(x,\pi_{\text{DN}})&=c_{\text{cm}}\frac{\frac{\mu_1}{\mu_1+\alpha} \left(
\frac{\mu_2}{\mu_2+\alpha}\i_{\{i=2\}}+\i_{\{i\neq2\}}
\right)
}{1-\frac{\mu_1}{\mu_1+\alpha}\frac{\mu_2}{\mu_2+\alpha}}\leq 
c_{\text{cm}}\frac{\frac{\mu_1}{\mu_1+\alpha}}{1-\frac{\mu_1}{\mu_1+\alpha}\frac{\mu_2}{\mu_2+\alpha}}. 
\end{align*}
The proof for the long-run average cost follows by employing a simple renewal argument.

\subsection{Proof of Step 3} \label{Proof of Lemma 10.4.2}
Choose $x=(i,j,t) \in\mathcal{S}$. Let $T_{x,1}$ denote the first passage time from state $x$ to state $(1,\text{SO},0)$, and $\tilde{F}_{x,1}(\alpha)=\mathbb{E}[e^{-\alpha T_{x,1}}]$ and $\tau_{x,1}=\mathbb{E}[T_{x,1}]$. 

Starting from state $x$, after time $t\in[0,\tau)$, the system is in an SO state with probability 1. More concretely, under the optimal policy (which is deterministic stationary), say $\pi^{(\infty)}_\alpha$, starting in state $x=(i,j,t)$, it will end up in state $(1,\text{SO},0)$ after time $t$ with probability $p_x$, and in state $(2,\text{SO},0)$ with probability $1-p_x$. In case state $x$ coincides with an SO state then $p_x=0$ or $p_x=1$. Once in an SO state, the system state observed at only the SO epochs behaves like a discrete time (irreducible and aperiodic) Markov chain with only states $(1,\text{SO},0)$ and $(2,\text{SO},0)$. Thus, $\tau_{x,1}=\mathbb{E}[T_{x,1}]=\lim_{\alpha\to0^+}\frac{1-\tilde{F}_{x,1}(\alpha)}{\alpha}<\infty$.

From the above
\begin{align*}
V_\alpha(x)&=\mathbb{E}_{x}^{\pi^{(\infty)}_\alpha}[\alpha\text{-cost from $x$ to state $(1,\text{SO},0)$ in $T_{x,1}$}] +\mathbb{E}_{x}^{\pi_\alpha^{(\infty)}}[e^{-\alpha T_{x,1}}]V_\alpha(1,\text{SO},0).
\end{align*}
Note that  $\mathbb{E}_{x}^{\pi^{(\infty)}_\alpha}[\alpha\text{-cost from $x$ to state $(1,\text{SO},0)$ in $T_{x,1}$}]$ is equal to: (1) the  expected discounted cost incurred directly in state $x$, which is upper bounded by $c_{\text{cm}}$, (2) the total expected discounted cost of all the SOs that occur in time $T_{x,1}$, which is upper bounded by $c_{\text{cm}} \tau_{x,1}$, (3) the total expected discounted cost of all the USOs that occur in time $T_{x,1}$, which is upper bounded by $c_{\text{cm}}\lambda \tau_{x,1}$, and (4) the total expected discounted cost of all the corrective maintenance opportunities that occur in time $T_{x,1}$, which is upper bounded by $c_{\text{cm}}(\mu_1+\mu_2) \tau_{x,1}$. All in all, 
\begin{align*}
\mathbb{E}_{x}^{\pi^{(\infty)}_\alpha}[\alpha\text{-cost from $x$ to state $(1,\text{SO},0)$ in $T_{x,1}$}]
\leq  c_{\text{cm}}(1+\tau_{x,1} + (\lambda +\mu_1+\mu_2)\tau_{x,1}).
\end{align*}
Then, straightforward computations yield
\begin{align*}
\left| V_\alpha(x)-V_\alpha(1,\text{SO},0)\right|
&= \Big|
\mathbb{E}_{x}^{\pi^{(\infty)}_\alpha}[\alpha\text{-cost from $x$ to state $(1,\text{SO},0)$ in $T_{x,1}$}]\\
&\qquad+
\mathbb{E}_{x}^{\pi_\alpha^{(\infty)}}[e^{-\alpha T_{x,1}}]V_\alpha(1,\text{SO},0)
-V_\alpha(1,\text{SO},0) \Big|\\
&\leq 
\mathbb{E}_{x}^{\pi^{(\infty)}_\alpha}[\alpha\text{-cost from $x$ to state $(1,\text{SO},0)$ in $T_{x,1}$}] +
\left| 1-\mathbb{E}_{x}^{\pi_\alpha^{(\infty)}}[e^{-\alpha T_{x,1}}]\right|V_\alpha(1,\text{SO},0)\\
&\leq c_{\text{cm}}(1+\tau_{x,1}+(\lambda +\mu_1+\mu_2)\tau_{x,1})+
\left( 1-\mathbb{E}_{x}^{\pi_\alpha^{(\infty)}}[e^{-\alpha T_{x,1}}]\right)V_\alpha(1,\text{SO},0).
\end{align*}
Similarly, for $z=(i',j',t')\in\mathcal{S}$,
\begin{align*}
\left| V_\alpha(z)- V_\alpha(1,\text{SO},0) ]\right|
&\leq c_{\text{cm}}(1+\tau_{x,1}+(\lambda +\mu_1+\mu_2)\tau_{z,1})+
\left( 1-\mathbb{E}_{z}^{\pi_\alpha^{(\infty)}}[e^{-\alpha T_{z,1}}]\right)V_\alpha(1,\text{SO},0).
\end{align*}
Then, 
\begin{align*}
\left| V_\alpha(x)-V_\alpha(z)\right|&\leq \left|V_\alpha(x)-V_\alpha(1,\text{SO},0)\right|
+\left|V_\alpha(z)-V_\alpha(1,\text{SO},0)\right|\\
&\leq 
 c_{\text{cm}}(2+\tau_{x,1}+\tau_{z,1}+(\lambda +\mu_1+\mu_2)(\tau_{x,1}+\tau_{z,1}))\\
 &\qquad +
 \left(1-\mathbb{E}_{x}^{\pi_\alpha^{(\infty)}}[e^{-\alpha T_{x,1}}]+1-\mathbb{E}_{z}^{\pi_\alpha^{(\infty)}}[e^{-\alpha T_{z,1}}]
 \right)
  V_\alpha(1,\text{SO},0).
\end{align*}
Combining the above with Step 1 yields
\begin{align*}
|V_\alpha(x)-V_\alpha(z)|&\leq 
c_{\text{cm}}(2+\tau_{x,1}+\tau_{z,1}+(\lambda +\mu_1+\mu_2)(\tau_{x,1}+\tau_{z,1}))\\
&\qquad +
 \left(1-\mathbb{E}_{x}^{\pi_\alpha^{(\infty)}}[e^{-\alpha T_{x,1}}]+1-\mathbb{E}_{z}^{\pi_\alpha^{(\infty)}}[e^{-\alpha T_{z,1}}]
 \right)
\frac{\frac{\mu_1}{\mu_1+\alpha}}{1-\frac{\mu_1}{\mu_1+\alpha}\frac{\mu_2}{\mu_2+\alpha}}. 
\end{align*}
Lastly, note that $ \left(1-\mathbb{E}_{x}^{\pi_\alpha^{(\infty)}}[e^{-\alpha T_{x,1}}]+1-\mathbb{E}_{z}^{\pi_\alpha^{(\infty)}}[e^{-\alpha T_{z,1}}]
 \right)\frac{\frac{\mu_1}{\mu_1+\alpha}}{1-\frac{\mu_1}{\mu_1+\alpha}\frac{\mu_2}{\mu_2+\alpha}}\leq (\tau_{x,1}+\tau_{z,1})\frac{\mu_1\mu_2}{\mu_1+\mu_2}$, which yields 
\begin{align*}
|V_\alpha(x)-V_\alpha(z)|&\leq 
 c_{\text{cm}}\left(2+\left(\lambda +\mu_1+\mu_2+1+\frac{\mu_1\mu_2}{\mu_1+\mu_2}\right)(\tau_{x,1}+\tau_{z,1})\right). 
\end{align*}

\subsection{Proof of Step 5} \label{Proof of Lemma 10.4.3}
To prove this step, we follow to a large extent the approach in \cite[Theorem 3.2 (c)--(d)]{Bhattacharya}. 
Consider the average optimality equality  \eqref{Eq:Bellman_Gen}, this yields for an arbitrary policy $\pi$, 
\begin{align*}
V(X_{k})&\leq c(X_{k};a_{k})+\mathbb{E}_x^\pi[V(X_{k+1})\, |\, X_{k},a_{k}]-g \,\tau(X_{k};a_{k}),\ k=0,1,\ldots,
\end{align*}
which can be equivalently written as
\begin{align*}
c(X_{k};a_{k})&\geq g \,\tau(X_{k};a_{k})+V(X_{k})-\mathbb{E}_x^\pi[V(X_{k+1})\, |\, X_{k},a_{k}],\ k=0,1,\ldots.
\end{align*}
Taking expectations on both sides one gets
\begin{align*}
\mathbb{E}_x^\pi[c(X_{k};a_{k})]&\geq g \,\mathbb{E}_x^\pi[\tau(X_{k};a_{k})]+\mathbb{E}_x^\pi[V(X_{k})]-\mathbb{E}_x^\pi[V(X_{k+1})],\ k=0,1,\ldots.
\end{align*}
Summing both sides of the above equation over $k = 0, 1, . . . , N - 1$, and dividing by $\mathbb{E}_x^\pi\left[\sum_{k=0}^{N-1}\tau(X_{k};a_{k})\right]$
one has
\begin{align}\label{Eq:g_Jxpi}
\frac{\mathbb{E}_x^\pi\left[\sum_{k=0}^{N-1}c(X_{k};a_{k})\right]}{\mathbb{E}_x^\pi\left[\sum_{k=0}^{N-1}\tau(X_{k};a_{k})\right]}
&\geq g +
\frac{V(x)-\mathbb{E}_x^\pi[V(X_{N})]}{\mathbb{E}_x^\pi\left[\sum_{k=0}^{N-1}\tau(X_{k};a_{k})\right]}.
\end{align}
Note that as $N\to\infty$, for $x=(i,j,t)$,
$$
\frac{t+(N-1)\tau}{1+\lambda\tau+\frac{\mu_1\mu_2}{\mu_1+\mu_2}\tau}
\leq \mathbb{E}_x^\pi\left[\sum_{k=0}^{N-1}\tau(X_{k};a_{k})\right].
$$ As such, $\mathbb{E}_x^\pi\left[\sum_{k=0}^{N-1}\tau(X_{k};a_{k})\right]$ is bounded from below for large values of $N$. Taking $\limsup\limits_{N\to\infty}$ on both sides of Equation \eqref{Eq:g_Jxpi} yields $J(x,\pi)\geq g$. 

Since, for $\pi^{*(\infty)}$, the above analysis holds with an equality, it is evident that  $J(x,\pi^{*(\infty)})= g$. 
Note that $g$ is an arbitrary limit point of $\alpha V_\alpha(x)$ as $\alpha\to0^+$. Furthermore, since $\alpha |V_{\alpha}(x)-V_{\alpha}(z)|\to0$ as $\alpha\to0^+$
 for all $x$ and for all $z$, it is now evident that $g=\limsup_{\alpha\to0^+}\alpha V_\alpha(x)$ for all $x\in\mathcal{S}$.

\section{Average cost equalities -- Bellman equations}
\label{sec:BellmanEq}
We proceed writing down the average cost equalities for the model at hand, cf. Proposition \ref{prop:ACOI-full}. More concretely, for $t\in[0,\tau)$, let $V(i,j,t)$ be the value function when the state of the system is $(i,j,t)\in \mathcal{S}$. The average optimality equations read as follows:
\begin{align}
V(2,\text{SC},t) =& 0-g \int_{0}^{t} e^{-(\mu_1+\lambda)x}\dif x+ V(1,\text{SO},0)\int_{t}^{\infty}(\mu_1+\lambda)e^{-(\mu_1+\lambda)x}\dif  x \nonumber\\
&+ \int_{0}^{t} \left(\frac{\mu_1}{\mu_1+\lambda}V(1,\text{SC},t-x)+\frac{\lambda}{\mu_1+\lambda}V(1,\text{USO},t-x)\right) (\mu_1+\lambda)e^{-(\mu_1+\lambda)x}\dif x \nonumber\\
=& e^{-\del{\mu_1+\lambda} t} \del{ \int_{0}^{t}  \del{\mu_1 V(1,\text{SC},y)+\lambda V(1,\text{USO},y) - g}e^{\del{\mu_1+\lambda}y} \dif y + V(1,\text{SO},0)},\label{eq:V2sc}\\
V(1,\text{SC},t) =& c_c + e^{-\del{\mu_2+\lambda} t} \del{ \int_{0}^{t}  \del{\mu_2 V(2,\text{SC},y)+\lambda V(2,\text{USO},y) - g}e^{\del{\mu_2+\lambda}y} \dif y + V(2,\text{SO},0)},\\
V(2,\text{USO},t) =& \min\left\{
c_p^{\text{uso}}+e^{-\del{\mu_2+\lambda} t} \del{ \int_{0}^{t}  \del{\mu_2 V(2,\text{SC},y)+\lambda V(2,\text{USO},y) - g}e^{\del{\mu_2+\lambda}y} \dif y + V(2,\text{SO},0)}
;
\right.\nonumber\\
&\left.\qquad
e^{-\del{\mu_2+\lambda} t} \del{ \int_{0}^{t}  \del{\mu_2 V(2,\text{SC},y)+\lambda V(2,\text{USO},y) - g}e^{\del{\mu_2+\lambda}y} \dif y + V(2,\text{SO},0)}
\right\},  \label{eq:USO2}\\
V(2,\text{SO},0) =& \min\left\{
c_p^{\text{so}}+
e^{-\del{\mu_2+\lambda} \tau} \del{\int_{0}^{\tau} \del{\mu_2 V(2,\text{SC},y)+\lambda V(2,\text{USO},y) - g}e^{\del{\mu_2+\lambda}y} \dif y + V(2,\text{SO},0)}
\right.;\nonumber\\
&\left.\qquad {}
e^{-\del{\mu_2+\lambda} \tau} \del{\int_{0}^{\tau} \del{\mu_2 V(2,\text{SC},y)+\lambda V(2,\text{USO},y) - g}e^{\del{\mu_2+\lambda}y} \dif y + V(2,\text{SO},0)}
\right\},
\label{eq:SO2}\\
V(1,\text{USO},t) =& \min\left\{
c_p^{\text{uso}}+p e^{-\del{\mu_2+\lambda} t} \del{ \int_{0}^{t}  \del{\mu_2 V(2,\text{SC},y)+\lambda V(2,\text{USO},y) - g}e^{\del{\mu_2+\lambda}y} \dif y + V(2,\text{SO},0)} +
\right.\nonumber\\
&\left.\qquad q e^{-\del{\mu_1+\lambda} t} \del{ \int_{0}^{t}  \del{\mu_1 V(1,\text{SC},y)+\lambda V(1,\text{USO},y) - g}e^{\del{\mu_1+\lambda}y} \dif y + V(1,\text{SO},0)} 
\right.;\nonumber\\
&\left.\qquad
e^{-\del{\mu_1+\lambda} t} \del{ \int_{0}^{t}  \del{\mu_1 V(1,\text{SC},y)+\lambda V(1,\text{USO},y) - g}e^{\del{\mu_1+\lambda}y} \dif y + V(1,\text{SO},0)}
\right\}, \label{eq:USO1}\\
V(1,\text{SO},0) =& \min\left\{
c_p^{\text{so}}+ p
e^{-\del{\mu_2+\lambda} \tau} \del{\int_{0}^{\tau} \del{\mu_2 V(2,\text{SC},y)+\lambda V(2,\text{USO},y) - g}e^{\del{\mu_2+\lambda}y} \dif y + V(2,\text{SO},0)}
\right. + \nonumber\\
&\left.\qquad  q
e^{-\del{\mu_1+\lambda} \tau} \del{\int_{0}^{\tau} \del{\mu_1 V(1,\text{SC},y)+\lambda V(1,\text{USO},y) - g}e^{\del{\mu_1+\lambda}y} \dif y + V(1,\text{SO},0)}
\right.;\nonumber\\
&\left.\qquad {}
e^{-\del{\mu_1+\lambda} \tau} \del{\int_{0}^{\tau} \del{\mu_1 V(1,\text{SC},y)+\lambda V(1,\text{USO},y) - g}e^{\del{\mu_1+\lambda}y} \dif y + V(1,\text{SO},0)}
\right\}.
\label{eq:SO1}
\end{align}

In this paragraph, we explain in detail how Equation \eqref{eq:V2sc} is obtained. State $(2,\text{SC},t)$ is associated with only the  decision ``do nothing''. Therefore, there is no minimum operator appearing on the right hand side of Equation \eqref{eq:V2sc} and the corresponding cost is  equal to zero. For the other terms appearing on the right hand side of Equation \eqref{eq:V2sc}, it suffices to note that there are three possible evolutions in terms of the state of the system: either an SO or an SC or  a USO, where the time till the next SO is equal to $t$, while the times till the SC and USO are exponentially distributed with rates $\mu_1$ and $\lambda$, respectively.   In particular, the expected sojourn time of the semi-Markov decision process in state $(2,\text{SC},t)$ can be calculated as the expectation of the minimum of a deterministic time $t$ and two exponentially distributed times, which can be easily verified to be equal to $\int_{0}^{t} e^{-(\mu_1+\lambda)x}\dif x$.  
The set of optimality equations for the remaining states can be obtained using very similar arguments. Note that in Equations \eqref{eq:USO2}--\eqref{eq:SO1}, inside the minimum, the left term corresponds to the action `perform preventive maintenance', while the right terms  correspond to the action `do nothing'. 

We observe that, since $c_{\text{pm}}^{\text{so}}, c_{\text{pm}}^{\text{uso}}>0$ and $p+q=1$, Equations \eqref{eq:USO2} and \eqref{eq:SO2} yield that it is never optimal to perform preventive maintenance in state $2$ in both USOs and SOs, respectively.
 
 We define the following auxiliary functions, for $t\in[0,\tau)$,
\begin{equation}\label{eq:fis}
F_i(t) = e^{-\del{\mu_{i}+\lambda} t} \del{ \int_{0}^{t}  \del{\mu_i V(i,\text{SC},y)+\lambda V(i,\text{USO},y) - g}e^{\del{\mu_i+\lambda}y} \dif y
+ V(i,\text{SO},0)}, \ i \in \{1,2\},
\end{equation}
so that Equations \eqref{eq:V2sc}-\eqref{eq:SO1} reduce to
\begin{align}
V(1,\text{SC},t) =&\ c_{\text{cm}} + F_2(t),\quad V(2,\text{SC},t) = F_1(t), \quad t\in[0,\tau), \label{V2SCt}\\
V(i,\text{USO},t) =& \min\cbr{c_{\text{pm}}^{\text{uso}}+p  F_2(t) + q  F_i(t),F_i(t)}, \quad i\in\{1,2\}, t\in[0,\tau), \label{V1USOt}\\
V(i,\text{SO},0) =& \min\cbr{c_{\text{pm}}^{\text{so}}+p  F_2(\tau) + q   F_i(\tau) , F_i(\tau)}, \quad i\in\{1,2\}, t\in[0,\tau). \label{V1SOt}
\end{align}

\section{Proof of Theorem \ref{thm:opt1}}
\label{proofTheorem1}
\noindent \begin{proof} 
We distinguish four cases, each corresponding to a different set of actions. {Case (i):} $F_1(\tau)-F_2(\tau)\leq \frac{c_{\text{pm}}^{\text{so}}}{p}$; {Case (ii):} $\frac{c_{\text{pm}}^{\text{so}}}{p}< F_1(\tau)-F_2(\tau)< \frac{c_{\text{pm}}^{\text{uso}}}{p}$; {Case (iii):} $\frac{c_{\text{pm}}^{\text{uso}}}{p} < F_1(\tau)-F_2(\tau)$;  {Case (iv):} $F_1(\tau)-F_2(\tau)= \frac{c_{\text{pm}}^{\text{uso}}}{p}$.

\begin{description}
\item[Case (i):]  In state $(2,\text{SO},0)$, it is optimal to not perform preventive maintenance. Furthermore, from the assumption 
\begin{align}
F_1(\tau)-F_2(\tau) &\leq \frac{c_{\text{pm}}^{\text{so}}}{p} \label{AssumptionCondition}
\end{align} 
and Equation  \eqref{V1SOt} for $i=1$, it becomes evident that it is also optimal to not perform preventive maintenance in state $(1,\text{SO},0)$. Since the function $F_1(t)-F_2(t)$ is, by definition, a continuous function in $t\in[0,\tau]$, $c_{\text{pm}}^{\text{so}}<c_{\text{pm}}^{\text{uso}}$, and taking into account Equation \eqref{AssumptionCondition}, it is evident that there exists an $\varepsilon>0$, such that
\begin{align}\label{asumption2}
F_1(t)-F_2(t)  \leq  \frac{c_{\text{pm}}^{\text{uso}}}{p}, \text{ for all }t\in(\tau-\epsilon,\tau].
\end{align}
Equation \eqref{asumption2}, in light of Equation \eqref{V1USOt}, implies that if the elapsed time from the SO is less than $\epsilon$, then, under the assumption it is optimal to not perform preventive maintenance on the system in state $(i,\text{SO},0)$, it is also not optimal to perform preventive maintenance at a USO. In this case, for $t\in(\tau-\epsilon,\tau]$, we have that $V(1,\text{USO},t)=F_1(t)$ and $V(2,\text{USO},t) = F_2(t)$, cf. Equation \eqref{V1USOt}. Taking the derivative with respect to $t$ in Equation \eqref{eq:fis} and substituting the above obtained values for $V(1,\text{USO},t)$ and $V(2,\text{USO},t)$ yields
\begin{align*}
F_1'(t)-F_2'(t)&=-(\mu_1+\lambda)F_1(t)+\mu_1 V(1,\text{SC},t)+\lambda V(1,\text{USO},t)\nonumber\\
&\quad +(\mu_2+\lambda)F_2(t)-\mu_2 V(2,\text{SC},t)-\lambda V(2,\text{USO},t)\nonumber\\
&= -(\mu_1+\mu_2)F_1(t) + (\mu_1+\mu_2)F_2(t) + \mu_1 c_{\text{cm}},\ t\in(\tau-\epsilon,\tau].
\end{align*} 
The solution to the above differential equation reads
\begin{align}
F_1(t)-F_2(t) =& \frac{\mu_1c_{\text{cm}}}{\mu_1+\mu_2} + \left(F_1(\tau)-F_2(\tau) - \frac{\mu_1 c_{\text{cm}}}{\mu_1 + \mu_2} \right) e^{(\mu_1+\mu_2)(\tau-t)},\ t\in(\tau-\epsilon,\tau]. 
\label{solution:dif3}
\end{align}
If $F_1(\tau)-F_2(\tau) - \frac{\mu_1 c_{\text{cm}}}{\mu_1 + \mu_2}\neq 0$, it follows that, for $t\in(\tau-\epsilon,\tau]$, the function $F_1(t)-F_2(t)$ is strictly monotone. In this case, by extending the previous analysis to the entire domain, which would maintain the strict monotonicity of the function $F_1(t)-F_2(t)$, we would reach a contradiction: For $t=0$, Equation \eqref{eq:fis} yields $F_1(0)=V(1,\text{SO},0)\stackrel{\eqref{V1SOt}}{=}F_1(\tau)$ and $F_2(0)=V(2,\text{SO},0)\stackrel{\eqref{V1SOt}}{=}F_2(\tau)$, where $\stackrel{(\cdot)}{=}$ denotes that the equality follows from Equation ($\cdot$). We thus have
\begin{align}
F_1(0)-F_2(0) = F_1(\tau) -F_2(\tau). \label{asumption1}
\end{align}
Due to \eqref{asumption1}, it is evident that $F_1(\tau)-F_2(\tau) - \frac{\mu_1 c_{\text{cm}}}{\mu_1 + \mu_2}= 0$, thus the function $F_1(t)-F_2(t)$ satisfying Equation \eqref{solution:dif3} is a constant function, i.e. 
\begin{align}\label{conditionForConstant}
F_1(t)-F_2(t) = \frac{\mu_1c_{\text{cm}}}{\mu_1+\mu_2},\ t\in(0,\tau].
\end{align}
Combining Equation \eqref{AssumptionCondition} with Equation \eqref{conditionForConstant} leads to the optimality condition for {Case (i)}. That is, if $$\mu_1c_{\text{cm}}\leq (\mu_1+\mu_2)\frac{c_{\text{pm}}^{\text{so}}}{p},$$ we do not perform preventive maintenance at any opportunity. 
\item[Case (ii):]  In state $(2,\text{SO},0)$, similarly to the previous case, it is optimal to not perform preventive maintenance. However, from the assumption 
\begin{align}
\frac{c_{\text{pm}}^{\text{so}}}{p} < F_1(\tau)-F_2(\tau) < \frac{c_{\text{pm}}^{\text{uso}}}{p} \label{AssumptionCoditionCaseii}
\end{align} 
and Equation  \eqref{V1SOt} for $i=1$, it becomes evident that it is optimal to perform preventive maintenance on the system in state $(1,\text{SO},0)$. Similarly to Case (i), as $F_1(\tau)-F_2(\tau) < \frac{c_{\text{pm}}^{\text{uso}}}{p}$, there exists an  $\varepsilon>0$ for which \eqref{asumption2} holds.

Repeating the same analysis as in {Case (i)}, we can show that, for $t\in[0,\tau]$, the function $F_1(t)-F_2(t)$ satisfies Equation \eqref{solution:dif3} and that it is a non-decreasing function if 
\begin{align}
F_1(\tau)-F_2(\tau) - \frac{\mu_1c_{\text{cm}}}{\mu_1+\mu_2}<0 \label{conditionIncreasingness}.
\end{align}
However, for $t=0$, we now have that
\begin{align}
F_1(0) - F_2(0) &\stackrel{\eqref{eq:fis}}{=} V(1,SO,0) - V(2,SO,0) \nonumber \\ 
 &= c_{\text{pm}}^{\text{so}} + p F_2(\tau) + (1-p) F_1(\tau) -F_2(\tau) \nonumber \\ 
&=c_{\text{pm}}^{\text{so}} + (1-p)(F_1(\tau) - F_2(\tau)). \label{equalityCaseii}
\end{align}
Combining \eqref{equalityCaseii} with \eqref{solution:dif3} (on the domain $t\in[0,\tau]$) yields
\begin{align}
F_1(\tau)-F_2(\tau) = \frac{\left(1-e^{(\mu_1+\mu_2)\tau} \right) \frac{\mu_1 c_{\text{cm}}}{\mu_1 + \mu_2}  - c_{\text{pm}}^{\text{so}}}{1-p- e^{(\mu_1+\mu_2)\tau} } .\label{conditionCaseii}
\end{align}
Combining Equations \eqref{AssumptionCoditionCaseii},  \eqref{conditionIncreasingness}, and  \eqref{conditionCaseii} leads to the optimality condition for {Case (ii)}. That is, if $$(\mu_1+\mu_2)\frac{c_{\text{pm}}^{\text{so}}}{p} < \mu_1 c_{\text{cm}} < \left( \frac{c_{\text{pm}}^{\text{uso}}}{p} - \frac{c_{\text{pm}}^{\text{uso}} - c_{\text{pm}}^{\text{so}}}{e^{(\mu_1 +\mu_2)\tau}-1} \right) (\mu_1+\mu_2),$$ we perform preventive maintenance on the system if it is in state 1 at an SO but not at a USO. 

\item[Case (iii):]   In state $(2,\text{SO},0)$, similarly to the previous case, it is optimal to not perform preventive maintenance. However, from the assumption 
\begin{align}
F_1(\tau)-F_2(\tau) > \frac{c_{\text{pm}}^{\text{uso}}}{p}> \frac{c_{\text{pm}}^{\text{so}}}{p} \label{AssumptionCoditionCaseiii}
\end{align}
and Equation  \eqref{V1SOt} for $i=1$, it becomes evident that it is optimal to perform preventive maintenance on the system in state $(1,\text{SO},0)$. Along the lines of the previous cases, as $F_1(\tau)-F_2(\tau) > \frac{c_{\text{pm}}^{\text{uso}}}{p}$, there exists an  $\varepsilon>0$ for which 
\begin{align}\label{asumption3}
F_1(t)-F_2(t)  \geq  \frac{c_{\text{pm}}^{\text{uso}}}{p}, \text{ for all }t\in(\tau-\epsilon,\tau].
\end{align}
In this case, for $t\in(\tau-\epsilon,\tau]$, we have that $V(1,\text{USO},t)=c_{\text{pm}}^{\text{uso}} + pF_2(t) + (1-p) F_1(t)$ and $V(2,\text{USO},t) = F_2(t)$ (cf. Equation \eqref{V1USOt}). Taking a derivative with respect to $t$ in \eqref{eq:fis} and substituting the above obtained values for $V(1,\text{USO},t)$ and $V(2,\text{USO},t)$ yields
\begin{align}
F_1'(t)-F_2'(t)=&-(\mu_1+\lambda)F_1(t)+\mu_1 V(1,\text{SC},t)+\lambda V(1,\text{USO},t)\nonumber\\
& +(\mu_2+\lambda)F_2(t)-\mu_2 V(2,\text{SC},t)-\lambda V(2,\text{USO},t) \nonumber \\ 
=&  -(\mu_1+\lambda)F_1(t) + \mu_1 (c_{\text{cm}} + F_2(t)) + \lambda (c_{\text{pm}}^{\text{uso}} + pF_2(t) + (1-p) F_1(t))\nonumber\\
&+(\mu_2+\lambda)F_2(t) - \mu_2 F_1(t) - \lambda F_2(t) \nonumber \\
=& -(\mu_1+\mu_2 + \lambda p) (F_1(t)-F_2(t)) + \mu_1 c_{\text{cm}} + \lambda c_{\text{pm}}^{\text{uso}},\ t \in (\tau-\varepsilon,\tau].\label{eq:difeqscaseiii}
\end{align} 
The solution to the above differential equation reads 
\begin{align}
F_1(t)-F_2(t) =& \frac{\mu_1c_{\text{cm}} + \lambda c_{\text{pm}}^{\text{uso} }}{\mu_1+\mu_2 + \lambda p } + \left(F_1(\tau)-F_2(\tau) -\frac{\mu_1c_{\text{cm}} + \lambda c_{\text{pm}}^{\text{uso} }}{\mu_1+\mu_2 + \lambda p } \right) e^{(\mu_1+\mu_2 + \lambda p )(\tau-t)},\ t \in (\tau-\varepsilon,\tau].
\label{solution:dif8}
\end{align}
Note that, if we assume that $F_1(\tau)-F_2(\tau) -\frac{\mu_1c_{\text{cm}} + \lambda c_{\text{pm}}^{\text{uso} }}{\mu_1+\mu_2 + \lambda p }\geq 0$, then we can extend \eqref{solution:dif8} on the entire domain $t\in[0,\tau]$, and the function $F_1(t)-F_2(t)$ is non-increasing. However, this is unfeasible. Note that, for $t=0$, Equation \eqref{eq:fis} yields $F_1(0)=V(1,\text{SO},0)\stackrel{\eqref{V1SOt}}{=}c_{\text{pm}}^{\text{so} }+pF_2(\tau)+qF_1(\tau)$ and $F_2(0)=V(2,\text{SO},0)\stackrel{\eqref{V1SOt}}{=}F_2(\tau)$, thus
\begin{align}\label{DiffAtZero_caseiii}
F_1(0)-F_2(0) = c_{\text{pm}}^{\text{so} }+q(F_1(\tau) -F_2(\tau))\geq F_1(\tau) -F_2(\tau) \ \Leftrightarrow\  F_1(\tau)-F_2(\tau) \leq \frac{c_{\text{pm}}^{\text{so}}}{p},
\end{align}
which contradicts Assumption \eqref{AssumptionCoditionCaseiii}. Due to this contradiction, it is necessary to assume that $F_1(\tau)-F_2(\tau) -\frac{\mu_1c_{\text{cm}} + \lambda c_{\text{pm}}^{\text{uso} }}{\mu_1+\mu_2 + \lambda p }< 0$. This implies that the function $F_1(t)-F_2(t)$ is non-decreasing and we can extend \eqref{solution:dif8} on the  domain $t\in[t^*,\tau]$, where $t^*$  is such that $F_1(t^*)-F_2(t^*)= \frac{c_{\text{pm}}^{\text{uso}}}{p}$, i.e.  
\begin{align}
F_1(t)-F_2(t) =& \frac{\mu_1c_{\text{cm}} + \lambda c_{\text{pm}}^{\text{uso} }}{\mu_1+\mu_2 + \lambda p } + \left(F_1(\tau)-F_2(\tau) -\frac{\mu_1c_{\text{cm}} + \lambda c_{\text{pm}}^{\text{uso} }}{\mu_1+\mu_2 + \lambda p } \right) e^{(\mu_1+\mu_2 + \lambda p )(\tau-t)},\ t \in [t^*,\tau].
\label{solution:dif8b}
\end{align}
See Figure \ref{fig:case3} for a visualization of $F_1(t)-F_2(t) $.

\begin{figure}[!ht]
\centering
  \includegraphics[width=0.7\linewidth]{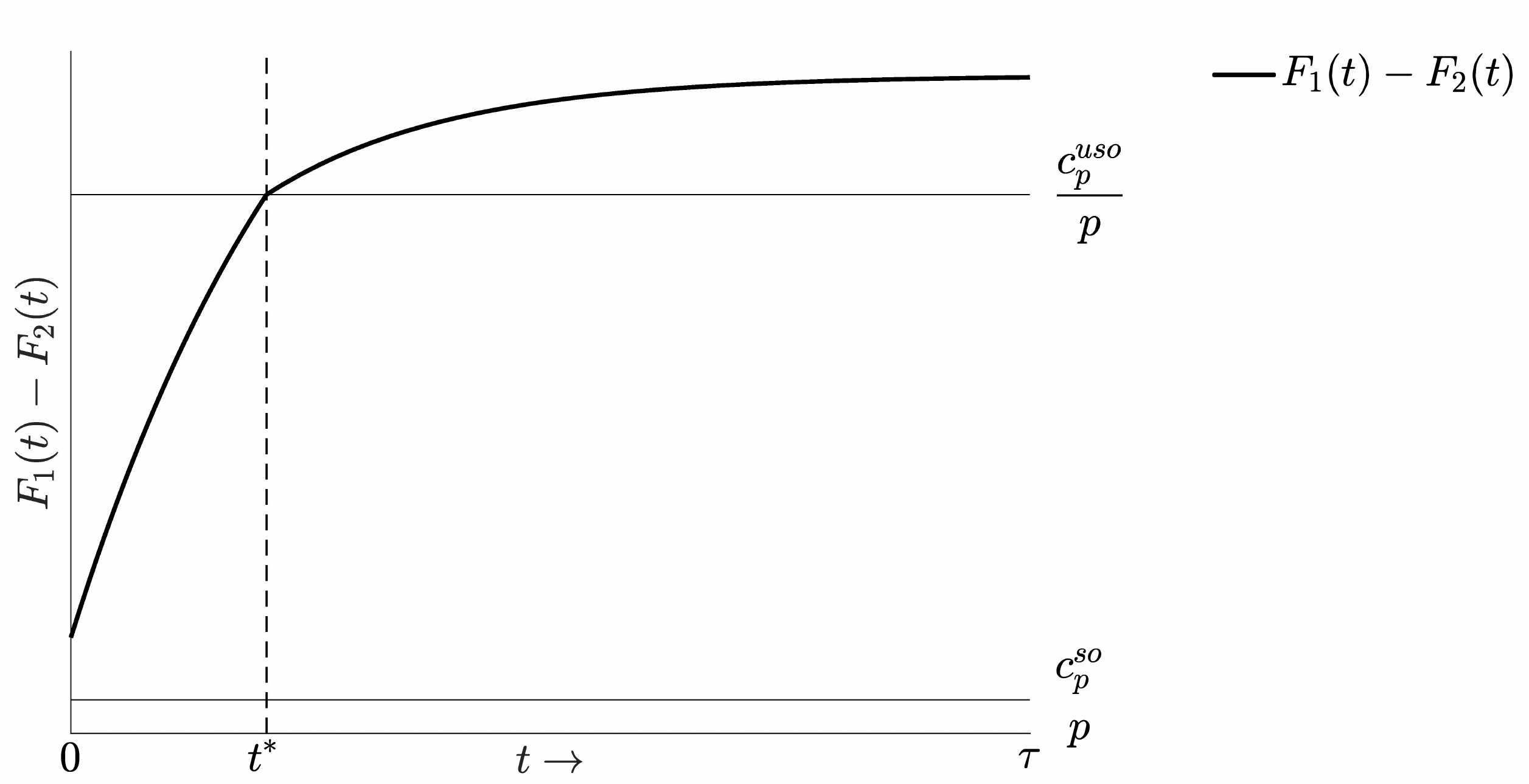}
  \caption{The case of maintaining a system at scheduled and unscheduled opportunities in $t\in[{t^*},\tau)$.}
  \label{fig:case3}
\end{figure}

From the definition of $t^*$, and the continuity of $F_1(t)-F_2(t) $, it follows that there exists an $\varepsilon>0$, such that
\begin{align}\label{asumption_2}
F_1(t)-F_2(t)  \leq  \frac{c_{\text{pm}}^{\text{uso}}}{p}, \text{ for all }t\in(t^*-\epsilon,t^*].
\end{align}
Note that if one were to assume that $F_1(t)-F_2(t)  \geq  \frac{c_{\text{pm}}^{\text{uso}}}{p}$, for all $t\in(t^*-\epsilon,t^*]$, then due to Equation \eqref{DiffAtZero_caseiii}, this would again contradict Assumption \eqref{AssumptionCoditionCaseiii}.

Now repeating the analysis performed in Case (i), albeit in a different domain, we can show that, for $t \in [0,t^*]$, 
\begin{align}
F_1(t)-F_2(t) =& \frac{\mu_1c_{\text{cm}}}{\mu_1+\mu_2} + \left(F_1(t^*)-F_2(t^*) - \frac{\mu_1 c_{\text{cm}}}{\mu_1 + \mu_2} \right) e^{(\mu_1+\mu_2)(t^*-t)},\ t\in [0,t^*].
\label{solution:dif7}
\end{align}
From the continuity of $F_1(t)-F_2(t)$ at $t=t^*$, we obtain
\begin{equation} 
\frac{c_{\text{pm}}^{\text{uso}}}{p} = \frac{\mu_1c_{\text{cm}} + \lambda c_{\text{pm}}^{\text{uso} }}{\mu_1+\mu_2 + \lambda p } + \left(F_1(\tau)-F_2(\tau) -\frac{\mu_1c_{\text{cm}} + \lambda c_{\text{pm}}^{\text{uso} }}{\mu_1+\mu_2 + \lambda p } \right) e^{(\mu_1+\mu_2 + \lambda p )(\tau-t^*)}. \label{equation1} 
\end{equation}
Furthermore, setting $t=0$ in Equation \eqref{solution:dif7} and using \eqref{DiffAtZero_caseiii} yields
\begin{equation}
c_{\text{pm}}^{\text{so}} + (1-p)(F_1(\tau) - F_2(\tau)) =  \frac{\mu_1c_{\text{cm}}}{\mu_1+\mu_2} + \left( \frac{c_{\text{pm}}^{\text{uso}}}{p} - \frac{\mu_1 c_{\text{cm}}}{\mu_1 + \mu_2} \right) e^{(\mu_1+\mu_2)t^*}. \label{equation2}
\end{equation}
Note that Equations \eqref{equation1} and \eqref{equation2} form a system of two equations with two unknowns, which produce a unique solution for  $t^*$, cf.  Equation \eqref{thm1eq1}.  Since $F_1(t) - F_2(t)$ is a continuous function throughout $[0,\tau)$, we can directly use the optimality condition for {Case (ii)} to state the optimality condition for this case. That is, if $$\mu_1 c_{\text{cm}} > \left( \frac{c_{\text{pm}}^{\text{uso}}}{p} - \frac{c_{\text{pm}}^{\text{uso}} - c_{\text{pm}}^{\text{so}}}{e^{(\mu_1 +\mu_2)\tau}-1} \right) (\mu_1+\mu_2),$$ we perform preventive maintenance on the system if it is in state 1 at an SO and at a USO for which the residual time until the next SO is in the interval $[\hat{t},\tau)$, with $\hat{t} = \min\{\tau,\max\{0,t^*\}\}$.
\item[Case (iv):]  This case follows evidently by performing again the steps of Case (iii) for  $t^*=\tau$.\end{description}
\end{proof}

\section{Proof of Theorem \ref{prop:unequal}}
\label{proofTheorem2}
\noindent \begin{proof}
Similarly to the proof of Theorem \ref{thm:opt1}, we need to make certain assumptions here regarding the actions at the given opportunities. In particular, we distinguish four cases, each corresponding to a different set of actions:
{Case (i):} $F_1(\tau)-F_2(\tau)\leq \frac{c_{\text{pm}}^{\text{uso}}}{p}$; {Case (ii):} $\frac{c_{\text{pm}}^{\text{uso}}}{p}< F_1(\tau)-F_2(\tau)< \frac{c_{\text{pm}}^{\text{so}}}{p}$; {Case (iii):} $\frac{c_{\text{pm}}^{\text{so}}}{p} < F_1(\tau)-F_2(\tau)$;  {Case (iv):} $F_1(\tau)-F_2(\tau)= \frac{c_{\text{pm}}^{\text{so}}}{p}$.
The proof of this theorem is similar in structure to the proof of Theorem \ref{thm:opt1} and for this reason it is omitted.
\end{proof}

\section{Proof of Theorem \ref{them:costsDeferring}}
\label{proofTheorem4}
\noindent \begin{proof} 
We first focus on the derivation of the cycle length appearing in the denominator of Equation \eqref{Eq:longrunrateofcostdef}. Observe that the length of a renewal cycle consists of the time the system spends in state 2 plus the time from the state-change $2 \to 1$ until the first successful maintenance. To this purpose, let $C\!L$ denote the length of the part of the renewal cycle that the underlying stochastic process spends in state 1. Furthermore, let $Y$ denote the random amount of time from a state-change $2\to 1$ to the first SO, we then have for the probability density function of $Y$  that $$f_Y(y)= f_{T_{\mu_2}}(\tau-y|T_{\mu_2}<\tau),$$ which leads to Equation \eqref{residualTime_thm}. 
Conditioning on $Y$, a renewal cycle can either end before the first SO, or at the first SO, or after the first SO. Hence, we have that the expected cycle length is equal to
\begin{align}
\frac{1}{\mu_2} + \E\left[C\!L\,  \i_{\{C\!L\, \leq Y\}}\right] + \E\left[C\!L\,  \i_{\{C\!L\, >Y\}}\right].\label{ECL}
\end{align}

We first focus on deriving expressions for the individual expectations in Equation \eqref{ECL}. Note that the first successful maintenance can be of type $j \in \{\text{SC},\text{SO},\text{USO}\}$ and may occur in the interval $[t, t']$, this is in short denoted by $j\,[t, t']$. Thus, rewriting the first part in Equation \eqref{ECL} results in (cf. Equation \eqref{eclFirstPart_thm})
\begin{align}
\E\left[C\!L\  \i_{\{C\!L\,  \leq Y\}} \right]  &= \E\left[C\!L  \i_{\{\text{USO}{[\tau- Y,\tau-\tilde{t}]}\}}\right] + \E\left[C\!L \i_{\{\text{SO}{[\tau-Y,\tau]}\}}\right] + \E\left[C\!L  \i_{\{\text{CM}{[\tau- Y,\tau]}\}}\right]. \label{eclFirstPart}
\end{align}
For the second expectation in Equation \eqref{ECL}, observe that the length of this part can be further decomposed: first the system goes through a geometric number of intervals of length $\tau$ in which no successful maintenance activity takes place, after which the system enters the last interval in which the successful maintenance activity takes place. To this end, let $p_u$ be the probability that there is no successful maintenance activity in an arbitrary interval between two SOs (including the SO with which this interval ends) after the state change $2\rightarrow 1$, i.e. 
$$p_u :=(1-p) \P[T_{\mu_1} > \tau, T_{\lambda p} > \tau-\tilde{t}]  
 = (1-p) e^{-\mu_1\tau - \lambda p (\tau - \tilde{t})} = (1-p)  \P\left[\text{SO}{[0,\tau]} \right].$$ 
We then have, from the memoryless property of $T_{\mu_1}$ and $T_{\lambda p}$,
\begin{align}
\E\left[C\!L\,  \i_{\{C\!L\, >Y\}}\right] &= (1-p)\P [ \text{SO}_{[\tau-Y,\tau]} ] \left( \E\left[Y\right] + \sum_{k=0}^{\infty} p_u^k(1-p_u) \left( \E\left[C\!L\,  \i_{\{Y+k\tau \leq C\!L\,  \leq Y + (k+1) \tau\}} \right] \right) \right) \nonumber \\
&= (1-p)\P [ \text{SO}_{[\tau-Y,\tau]} ]\left( \E\left[Y\right]  + \sum_{k=0}^{\infty} p_u^k (1-p_u) \left( k\tau +  \E\left[C\!L'\,  \i_{\{C\!L'\leq Y\}} \,|\, Y=\tau \right]    \right) \right)\nonumber \\
&= (1-p)\P [ \text{SO}_{[\tau-Y,\tau]} ] \left(  \E\left[Y\right]  + \frac{\tau p_u}{1-p_u} +  \E\left[C\!L'\,  \i_{\{C\!L'\leq Y\}} \,|\, Y=\tau \right]  \right), \label{EXP2}
\end{align}
where $\E\left[C\!L'\,  \i_{\{CL'\leq Y\}} \,|\, Y=\tau \right]$ is the expected length of the last part of the renewal cycle, i.e. the interval in which the successful maintenance activity takes place. Analogously to Equation \eqref{eclFirstPart}, we can further decompose  $\E\left[C\!L'\,  \i_{\{C\!L'\leq Y\}} \,|\, Y=\tau \right]$ by conditioning on the type of the successful maintenance activity with which it ends.

We are now left with defining the events that lead to $j\,[t, t']$, such that we can calculate the expectations in Equations \eqref{USOCE_thm}-\eqref{CMCE_thm}. With respect to $\text{SO}[\tau-y, \tau]$, observe that if $y\in[0,\tilde{t})$,  $\i_{\{ \text{SO}{[\tau-y,\tau]}\}}$ is equal to 1 if $T_{\mu_1}>y$, since we do not take any USOs. If  $y\in[\tilde{t},\tau]$, no successful USOs in $[\tau-y,\tau-\tilde{t}]$ can occur and $T_{\mu_1}>y$ for $\i_{\{ \text{SO}{[\tau-y,\tau]}\}}$ to be equal to 1. Combining this leads to Equation \eqref{ISO_thm}. Equations \eqref{IUSO_thm} and \eqref{ICM_thm} are obtained along similar lines. Note that all expectations and probabilities only involve exponentially distributed random variables. Consequently, closed-form expressions can be obtained using straightforward calculus. However, for the sake of brevity, we have chosen to provide one closed-form expression and omit the rest (which can be obtained analogously).  For Equation \eqref{USOCE_thm}, we have for $y>\tilde{t}$:
\begin{align}
\E\left[CL \i_{\{\text{USO}{[\tau- y,\tau-\tilde{t}]}\}}\right] &= \E\left[ T_{\lambda p} \i_{\{\text{USO}{[\tau- y,\tau-\tilde{t}]}\}}   \right] \nonumber\\
&= \E\left[ T_{\lambda p} \i_{\{T_{\lambda p} \leq \min\left\{y-\tilde{t},T_{\mu_1}\right\}\}}    \right]\nonumber\\  
&=  \int_{0}^{y-\tilde{t}} \E\left[ T_{\lambda p} \i_{\{T_{\lambda p} \leq x\}} \right] \mu_1 e^{-\mu_1 x} \dif x  +   \int_{y-\tilde{t}}^{\infty} \E\left[ T_{\lambda p} \i_{\{T_{\lambda p} \leq y - \tilde{t}\}}  \right]  \mu_1 e^{-\mu_1 x} \dif x \nonumber \\
&= \int_{0}^{y-\tilde{t}} \int_{0}^{x} z \lambda p e^{-\lambda p z} \dif z \mu_1 e^{-\mu_1 x} \dif x +   \int_{y-\tilde{t}}^{\infty}  \int_{0}^{y-\tilde{t}} z \lambda p e^{-\lambda p z} \dif z \mu_1 e^{-\mu_1 x} \dif x \nonumber \\
&= \frac{\lambda p }{\lambda p + \mu_1} \left( \frac{1- e^{-(\lambda p + \mu_1)(y-\tilde{t}) }(1+(\lambda p + \mu_1)(y-\tilde{t}))}{\lambda p + \mu_1} \right).  \nonumber
\end{align}

We now focus on the numerator of Equation \eqref{Eq:longrunrateofcostdef}, i.e. the expected cycle cost. To that end, let $C\!C$ be the cost incurred in a renewal cycle. The analysis for the expected cycle cost, $\E\left[C\!C\right]$, is similar to the analysis of the expected cycle length. Again, we decompose the length of a renewal cycle into three parts (i.e. the interval after the state change until the first SO, the geometric number of intervals of length $\tau$ in which no successful maintenance activity takes place, and the last interval in which the successful maintenance activity takes place), and compute the conditional expected cycle costs in these parts (mainly consisting of costs incurred at unsuccessful maintenance activities). Thus,
\begin{align} 
\E\left[C\!C\right] &= \E\left[C\!C\i_{\{C\!L\,  \leq Y\}}\right]+ \E\left[C\!C\i_{\{C\!L\,  > Y\}} \right] .\label{ECCfirst}
\end{align}
We first focus on the first part in Equation \eqref{ECCfirst} and condition further on the type of activity, which yields 
\begin{align}
 \E\left[C\!C\, \i_{\{C\!L\,  \leq Y\}}\right] &= \E\left[C\!C\,\i_{\{\text{USO}{[\tau-Y,\tau-\tilde{t}]}\}}\right] + \E\left[ C\!C\, \i_{\{ \text{SO}{[\tau-Y,\tau]}\}}\right] + \E\left[C\!C\,\i_{\{\text{C\!M}{[\tau-Y,\tau]}\}}\right].\nonumber
\end{align}
Analogous to the expected cycle lengh, the expected cost incurred during the geometric number of intervals of length $\tau$, in which no successful maintenance activity takes place, is equal to  
$$\sum_{k=0}^{\infty} p_u^k (1-p_u)k  \left(\lambda (1-p)(\tau -\tilde{t})c_{\text{pm}}^{\text{uso}} +c_{\text{pm}}^{\text{so}} \right) = \frac{(\lambda (1-p)(\tau -\tilde{t})c_{\text{pm}}^{\text{uso}} +c_{\text{pm}}^{\text{so}} ) p_u }{1-p_u }.$$ 
Observe that the expected cost in the interval in which the successful maintenance activity takes place is composed of two parts regardless of the type of activity, i.e. the cost of the successful maintenance activity itself and the cost related to the unsuccessful USOs up to the successful maintenance activity (see Equations \eqref{costUSO} - \eqref{costCM}).  Again, all expectations and probabilities related to the costs only involve exponentially distributed random variables, and again, for the sake of brevity, we have chosen to provide one closed-form expression and omit the rest (which can be obtained analogously). For Equation \eqref{costSO}, we have
\begin{align}
\E\left[CC\,\i_{\{\text{SO}{[\tau-y,\tau]}\}}\right] &= \Big(c_{\text{pm}}^{\text{so}} + \lambda(1-p) c_{\text{pm}}^{\text{uso}}  \max\left\{y-\tilde{t},0\right\} \Big)\P\left[\text{SO}{[\tau-y,\tau]}\,\right], \nonumber
\end{align}
with
\begin{align}
\P \left[  \text{SO}{[\tau-y,\tau]} \right] &= \P \left[ T_{\mu_1} > y \right]\i_{\{ y < \tilde{t}\}} + \P \left[ T_{\lambda p} > y-\tilde{t} , T_{\mu_1} > y \right]\i_{\{y\geq\tilde{t}\}} \nonumber\\
&=  e^{-\mu_1 y }\,\i_{\{y < \tilde{t}\}} +  e^{-(\mu_1 y + \lambda p (y-\tilde{t}))}\,\i_{\{y\geq \tilde{t}\}},  \nonumber
\end{align}
which completes the proof.
\end{proof}
\end{document}